\documentclass[11pt]{amsart}
\usepackage[utf8]{inputenc}
\usepackage{graphicx,color,verbatim}
\usepackage{amscd,amsmath,amsfonts,amssymb}

\usepackage[foot]{amsaddr}

\usepackage{hyperref,url}

\newtheorem{theorem}{Theorem}
\newtheorem{lemma}[theorem]{Lemma}
\newtheorem{proposition}[theorem]{Proposition}
\newtheorem{remark}[theorem]{Remark}

\newtheorem{definition}{Definition}

\numberwithin{equation}{section}

\newcommand{\R}{\mathbb R}
\newcommand{\N}{\mathbb N}

\newcommand{\Tr}{\text{Tr}_\Omega}
\newcommand{\SD}{\Sigma_{\mathcal D}}

\newcommand{\SN}{\Sigma_{\mathcal N}}
\newcommand{\X}{\mathcal{X}_{\Sigma_{\mathcal{D}}^*}^s(\mathcal{C}_{\Omega})}
\newcommand{\red}[1]{\textcolor{red}{#1}} 

\makeatletter
\@namedef{subjclassname@2020}{\textup{2020} Mathematics Subject Classification}
\makeatother

\begin{document}
	
\title[Positive solutions for a weighted mixed critical problem]{Positive solutions for a weighted critical problem with mixed boundary conditions}
	
\author{Alejandro Ortega}
\address[A. Ortega]{Departamento de Matemáticas Fundamentales, Facultad de Ciencias, Universidad Nacional de Educación a Distancia (UNED), 28040 Madrid, Spain}
\email{\tt alejandro.ortega@mat.uned.es}

\author{Luca Vilasi}
\address[L. Vilasi]{Department of Mathematical and Computer Sciences, Physical Sciences and Earth Sciences, University of Messina,
Viale F. Stagno d’Alcontres, 31 - 98166 Messina, Italy}
\email{\tt lvilasi@unime.it}
	
\author{Youjun Wang}
\address[Y. Wang]{Department of Mathematics, South China University of Technology, Guangzhou, 510640, China}
\email{\tt scyjwang@scut.edu.cn}


\thanks{Corresponding author: alejandro.ortega@mat.uned.es}

\begin{abstract}
We analyze the existence and multiplicity of positive solutions to a nonlocal elliptic problem involving the spectral fractional Laplace operator endowed with homogeneous mixed Dirichlet-Neumann boundary conditions and weighted critical nonlinearities. By means of variational methods and the Nehari manifold approach, we deduce the existence of multiple positive solutions under some assumptions on the behavior of the weight function around its maximum points. Such a behavior, formulated in terms of some rate growth, is explicitly determined and depends on the relation between the dimension, the order of the operator and the subcritical perturbation. In this way we extend and improve the results in \cite{lialiuzhatan2016existence}, dealing with the Dirichlet problem for the classical Laplace operator, to the nonlocal setting involving mixed boundary conditions.
\end{abstract}
	
\maketitle
	
\section{The problem}
In this paper we analyze the existence and multiplicity of solutions to the nonlocal problem
\begin{equation}\label{problem}\tag{$P_{\lambda,q}$}
	\left\{
	\begin{array}{ll} 
(-\Delta)^s u = \lambda u^q + Q(x)u^{2^*_s-1} & \text{ in } \Omega,\smallskip\\
 u>0	 & \text{ in } \Omega,\smallskip\\
  B(u)=0 & \text{ on } \partial\Omega,
	\end{array}
	\right.
\end{equation}
where $\Omega\subset\R^N$ is a bounded domain with a smooth boundary, $N>2s$, $\frac{1}{2}<s<1$,  $q\in[1,2^*_s-1)$ and $2^*_s:=2N/(N-2s)$ is the fractional critical Sobolev exponent. The operator $(-\Delta)^s$ is the \textit{spectral} fractional Laplacian on $\Omega$ endowed with mixed Dirichlet-Neumann boundary conditions
$$
B(u):=u\chi_{\SD} +\frac{\partial u}{\partial\nu}\chi_{\SN},
$$
where $\nu$ is the outward unit normal to $\partial\Omega$, $\chi_A$ denotes the characteristic function of the set $A\subset\partial\Omega$, and $\Omega$ satisfies the following set of assumptions:
\begin{itemize}
	\item[$(\Omega_1)$] $\SD$ and $\SN$ are smooth $(N-1)$--dimensional submanifolds of $\partial\Omega$;
	\item[$(\Omega_2)$] $\SD$ is a closed manifold with positive $(N-1)$--dimensional Lebesgue measure, say $|\SD|=\mu\in(0,|\partial\Omega|)$;
	\item[$(\Omega_3)$] $\SD\cap\SN=\emptyset$, $\SD\cup\SN=\partial\Omega$ and $\SD\cap\overline{\Sigma}_\mathcal{N}=\Gamma$, where $\Gamma$ is a smooth $(N-2)$--dimensional submanifold of $\partial\Omega$.
\end{itemize}
On the weight $Q$, we assume that $Q\in C^0(\overline\Omega)$, $Q(x)>0$, and
\begin{itemize}
	\item[$(Q_1)$] there exists $a^0\in\SN$ such that 
	$$Q(a^0)=\max_{\overline{\Omega}}Q(x):= Q_{M}\qquad\text{and}\qquad
	Q(x)-Q(a^0)=o\left(|x-a^0|^\alpha\right)\ \ \text{as}\ x\to a^0,
	$$
$$
\text{for some } \alpha\in(0,N) \text{ in the cases }  N=2,3 \;\text{ and }\; 1<q\leq \displaystyle\frac{6s-N}{N-2s}, \qquad (i);
$$
for
$$
	\alpha=\left\lbrace 
	\begin{array}{llr}
		\displaystyle N-\frac{(N-2s)(q+1)}{2} & \text{ if } N=2,3 \;\text{ and }\; \displaystyle\frac{6s-N}{N-2s}<q<\frac{N+2s}{N-2s}, & (ii)\\ \smallskip
		\displaystyle N-\frac{(N-2s)(q+1)}{2} & \text{ if } N\geq 4 \;\text{ and }\; \displaystyle 1\leq q < \frac{N+2s}{N-2s}. & (iii) 
	\end{array}	
	\right.
	$$
\end{itemize}
In our approach to \eqref{problem}, weights leading to multiplicity of solutions will behave as follows, 
\begin{itemize}
	\item[$(Q_2)$] there exist $a^1,\ldots,a^k\in\SN$, strict global maximizers of $Q$, such that, for $ i=1,\ldots,k$,
	$$Q(a^i)=\max_{\overline{\Omega}}Q(x):= Q_{M}\qquad\text{and}\qquad Q(x)-Q(a^i)=o\left(|x-a^0|^\alpha\right)\ \ \text{as}\ x\to a^i,$$
where 
$$
	\alpha=\left\lbrace 
	\begin{array}{llr}
		\displaystyle N-\frac{(N-2s)(q+1)}{2} & \text{ for } N=2,3 \;\text{ and }\; \displaystyle\frac{6s-N}{N-2s}<q<\frac{N+2s}{N-2s}, & (i)\\[7pt] 
		\displaystyle N-\frac{(N-2s)(q+1)}{2} & \text{ for } N\geq 4 \;\text{ and }\; \displaystyle 1\leq q < \frac{N+2s}{N-2s}. & (ii)
	\end{array}	
	\right.
	$$
\end{itemize}

This work find its motivation in \cite{lialiuzhatan2016existence} which deals with the problem driven by the classical Laplacian operator under Dirichlet boundary conditions, and from which we borrowed our methods of investigation. 
However, the mixed boundary condition framework and the nonlocality of the operator in \eqref{problem} add several difficulties in the study, making it particularly interesting. In terms of the hypotheses, a meaningful difference with respect to \cite{lialiuzhatan2016existence} is the fact that  
the weight $Q$ is assumed to attain its maximum values on the Neumann boundary $\Sigma_{\mathcal{N}}$ instead of in the interior of the domain $\Omega$, according to the underlying concentration phenomena. In addition, we extend the range of $\alpha$ in the case $(Q_1)-(i)$, thus allowing more general weights $Q$. 

A typical issue in critical problems is the concentration phenomenon associated with the lack of compactness of the Sobolev embedding at the critical exponent which, in turns, is closely related to the nonexistence of solutions to the associated pure critical problem. Nevertheless, this nonexistence result has a positive counterpart which has been extensively exploited since the seminal work by Brezis and Nirenberg (cf. \cite{BN}): by considering a lower order perturbation, we can find an energy threshold $c^*>0$ below which the compactness holds and, next,  construct fibers of the associated energy functional with energy below $c^*$. In the classical or local setting, these fibers are found with the aid of the Aubin-Talenti instantons, namely
\begin{equation*}
	v_{\varepsilon,x_0}(x)=\frac{\varepsilon^{\frac{N-2}{2}}}{(\varepsilon^2+|x-x_0|^2)^{\frac{N-2}{2}}},\qquad \varepsilon>0,
\end{equation*}
which are extremal functions for the Sobolev inequality 
\[S(N)\|u\|_{L^{2^*}(\Omega)}^2\leq\|\nabla u\|_{L^2(\Omega)}^2,\quad \text{ where } \quad S(N)=\pi N(N-2)\left(\frac{\Gamma(\frac{N}{2})}{\Gamma(N)}\right)^{\frac{2}{N}}.\]
Thus, once we have found the critical threshold determining the compactness, we can center a function $v_{\varepsilon,x_0}(x)$ at a point $x_0\in\Omega$ and, by concentrating enough the function around $x_0$ (i.e., by taking $\varepsilon>0$ small enough), we can construct fibers for the associated energy funcional below $c^*$ and then conclude, for instance, by employing a mountain pass type argument.

Following this scheme, the existence (and multiplicity) of solutions to problem \eqref{problem}  strongly depends, therefore, on the behavior of the pure critical problem
\begin{equation}\label{crit_prob}
	\left\{
	\begin{array}{ll} 
(-\Delta)^s u = u^{2^*_s-1} & \text{ in } \Omega,\smallskip\\
 u>0	 & \text{ in } \Omega,\smallskip\\
 B(u)=0 & \text{ on } \partial\Omega.
	\end{array}
	\right.
\end{equation}
In the case of homogeneous Dirichlet boundary conditions ($\Sigma_{\mathcal{D}}=\partial\Omega$ and $\Sigma_{\mathcal{N}}=\emptyset$), problem \eqref{crit_prob} has no solution, i.e., the fractional Sobolev constant $S(s,N)$ (see \eqref{soboconst}) is not attained. This is in contrast to the mixed setting where there is a solution to \eqref{crit_prob} if the Dirichlet part is small enough (cf. \cite[Theorem 1.1-(iii)]{colort2019the}). As a consequence, the corresponding Sobolev constant $S(\Sigma_{\mathcal{D}})$ (see \eqref{defi_sob}) is attained. More precisely, by \cite[Proposition 3.6]{colort2019the}, we have 
\begin{equation}\label{mix_i_sob}
S(\Sigma_{\mathcal{D}})\leq 2^{-\frac{2s}{N}}S(s,N),
\end{equation}
and, by \cite[Theorem 2.9]{colort2019the}, if $S(\Sigma_{\mathcal{D}})<2^{-\frac{2s}{N}}S(s,N)$, then $S(\Sigma_{\mathcal{D}})$ is attained. Thus, as far as the existence of solutions to \eqref{problem}  is concerned, we are naturally led to distinguish two cases: either $S(\Sigma_{\mathcal{D}})=2^{-\frac{2s}{N}}S(s,N)$ or $S(\Sigma_{\mathcal{D}})<2^{-\frac{2s}{N}}S(s,N)$. 

In the first case, there is no solution to the critical problem \eqref{crit_prob} and we can construct fibering maps for the associated energy functional below a certain critical threshold. This will be done by using the fractional analogue to Aubin-Talenti instantons, namely
\begin{equation*}
	u_{\varepsilon,x_0}(x)=\frac{\varepsilon^{\frac{N-2s}{2}}}{(\varepsilon^2+|x-x_0|^2)^{\frac{N-2s}{2}}},\quad \varepsilon>0,
\end{equation*}
which are extremal functions for the fractional Sobolev inequality (see \eqref{sobolev}). To summarize, our approach works as follows: first, we find a critical threshold below which the compactness is guaranteed; next, we center an instanton   $u_{\varepsilon,x_0}(x)$ at a point $x_0\in\Sigma_{\mathcal{N}}$ where $Q$ attains a maximum, say $x_0=a^i$. Finally, thanks to the behavior of $Q$ around $a^i$, we can construct fibers whose energy is lower than the critical threshold. 

In this regard, compared again to \cite{lialiuzhatan2016existence}, we carry out the asymptotic analysis of the maximizers of such fibers, thought as functions of $\varepsilon$ and $\lambda$, providing explicit bounds for them  (see Lemma \ref{propvarphilambda} below), this fact being decisive in our arguments.

This is the underlying idea of our main result, Theorem \ref{thmexistence}, in which the existence of solutions is deduced for different ranges of $\lambda$ and $q$ and, according to assumption $(Q_1)$, for different rates of growth of $Q$ near $a^0$.

On the contrary, if $S(\Sigma_{\mathcal{D}})<2^{-\frac{2s}{N}}S(s,N)$, the situation is more restrictive as in this case we have neither explicit expressions nor explicit estimates for the extremal functions at which $S(\Sigma_{\mathcal{D}})$ is attained, as it happens with the Aubin-Talenti instantons (see \eqref{estimtruncextremals} and Lemma \ref{Lpnormextremals}). As a consequence, the argument employed for the former case no longer works and the linear case $q=1$ in \eqref{problem} has to be ruled out. 

Considering instead the issue of the multiplicity of solutions, following again \cite{lialiuzhatan2016existence}, we seek for minimizers of the energy functional $J_\lambda$ of the extended problem (see \eqref{extorigprobl}) on suitable submanifolds of its Nehari manifold, defined in terms of a barycenter map. As stated in Proposition \ref{tildemlambdai}, we are able to compare the infimum of $J_\lambda$ on one of such submanifolds with the critical threshold only for small values of $\lambda$ and this, again, forces us to exclude the inequality case in \eqref{mix_i_sob} (as illustrated in Section \ref{sec:existence}, in this setting we guarantee that the supremum of the fibering maps is below the critical compactness threshold only for large values of $\lambda$).

This work continues an intensive line of research dealing with the existence and multiplicity of solutions to semilinear concave-convex critical elliptic problems initiated with the seminal works by Brezis and Nirenberg \cite{BN} and  Ambrosetti, Brezis and Cerami \cite{Ambrosetti1994}. Therefore, it is complicated to give a complete list of references. Among the works addressing semilinear nonhomogeneous elliptic problems we refer to \cite{AAP2000, CZ1996, CPM2004, Escobar1987, Hirano2001, Li2010, Lin2012, Liao2018} for the classical Laplace operator under Dirichlet boundary conditions;  to \cite{Barrios2012, Barrios2015, Bhakta2017, Kuhestani2019} for nonlocal operators with Dirichlet boundary conditions, and to \cite{ort2023concave} for the critical concave-convex problem for the spectral fractional Laplacian under mixed Dirichlet-Neumann boundary conditions.

The paper is organized as follows. In Section \ref{funcsett} we establish the variational setting of \eqref{problem} and, in particular, reformulate it in terms of the $s$-harmonic extension, recalling some preliminary facts that will be needed in the rest of our paper. In Section \ref{sec:existence} we focus on the existence of solutions according to the scheme illustrated before and state and prove our main results, Theorem \ref{thmexistence} and its companion Theorem \ref{thmexistence2}. The final Section \ref{sec:multiplicity} is devoted to the multiplicity issue addressed in Theorem \ref{thmksolutions}; the notion of barycenter map and the minimization on submanifolds of the Nehari manifolds allow us to deduce the existence of as many solutions as the number of maximum points of the weight function $Q(x)$.
\section{Functional setting and preliminaries}\label{funcsett}
The fractional Laplace operator $(-\Delta)^s$ considered in this work is defined through the spectral decomposition of the classical Laplace operator under mixed boundary conditions. Let $(\lambda_i,\varphi_i)$ be the eigenvalues and the eigenfunctions (normalized with respect to the $L^2(\Omega)$-norm) of $-\Delta$ with homogeneous mixed Dirichlet-Neumann boundary conditions, respectively. Then $(\lambda_i^s,\varphi_i)$ are the eigenvalues and the eigenfunctions of $(-\Delta)^s$. Consequently, given two smooth functions 
$$
u_i(x)=\sum_{j\geq1}\langle u_i,\varphi_j\rangle_{2}\varphi_j(x), \quad i=1,2,
$$
where $\langle u,v\rangle_2=\displaystyle\int_{\Omega}uv\,dx$ is the standard scalar product on $L^2(\Omega)$, one has
\begin{equation}\label{pre_prod}
	\langle(-\Delta)^s u_1, u_2\rangle_{2} = \sum_{j\ge 1} \lambda_j^s\langle u_1,\varphi_j\rangle_{2} \langle u_2,\varphi_j\rangle_{2},
\end{equation}
that is, the action of the fractional operator on a smooth function $u_1$ is given by
\begin{equation*}
	(-\Delta)^su_1=\sum_{j\ge 1} \lambda_j^s\langle u_1,\varphi_j\rangle_{2}\varphi_j.
\end{equation*}
As a consequence, the {\it spectral} fractional Laplace operator $(-\Delta)^s$ is well defined in the following Hilbert space of functions that vanish on $\SD$:
\begin{equation*}
	H_{\Sigma_{\mathcal{D}}}^s(\Omega)=\left\{u=\sum_{j\ge 1} a_j\varphi_j\in L^2(\Omega):\ u=0\ \text{on }\Sigma_{\mathcal{D}},\ \|u\|_{H_{\Sigma_{\mathcal{D}}^s}(\Omega)}^2:=
	\sum_{j\ge 1} a_j^2\lambda_j^s<+\infty\right\}.
\end{equation*}
\begin{remark}
	{\rm By \cite[Theorem 11.1]{Lions1972}, if $0<s\le \frac{1}{2}$, then $H_0^s(\Omega)=H^s(\Omega)$ and, thus, also
		$H_{\Sigma_{\mathcal{D}}}^s(\Omega)=H^s(\Omega)$, while for $\frac 12<s<1$, $H_0^s(\Omega)\subsetneq H^s(\Omega)$. Therefore, the range $\frac 12<s<1$ ensures $H_{\Sigma_{\mathcal{D}}}^s(\Omega)\subsetneq H^s(\Omega)$ and it provides us with the appropriate functional space for the mixed boundary problem \eqref{problem}.
	}
\end{remark}
Taking \eqref{pre_prod} in mind, the norm $\left\| \cdot\right\|_{H_{\SD}^s(\Omega)}$ is induced by the scalar product
\begin{equation*}
	\langle u_1,u_2\rangle_{H_{\Sigma_{\mathcal{D}}}^s}=\left\langle (-\Delta)^su_1,u_2\right\rangle_{2}=\langle (-\Delta)^{\frac{s}{2}}u_1,(-\Delta)^{\frac{s}{2}}u_2\rangle_{2}=\left\langle u_1,(-\Delta)^su_2\right\rangle_{2},\quad 
\end{equation*}
for all $u_1,u_2\in H_{\SD}^s(\Omega)$. The above chain of equalities can be simply stated as an integration-by-parts like formula. Thus, we say that $u\in H_{\SD}^s(\Omega)$ is a weak solution to \eqref{problem} if,
		\begin{equation*}
			\int_\Omega (-\Delta)^{s/2}u (-\Delta)^{s/2} v dx =\int_\Omega Q(x) u^{2^*_s-1}v dx + \lambda\int_\Omega u^q v dx,\quad\text{for all}\ v\in H_{\SD}^s(\Omega).
		\end{equation*}
The energy functional $I_\lambda:H_{\Sigma_{\mathcal{D}}}^s(\Omega)\to\R$ associated with \eqref{problem} is given by
\begin{equation*}
	I_\lambda(u):=\frac12\int_{\Omega}|(-\Delta)^{\frac{s}{2}}u|^{2}dx-\frac{1}{2^*_s}\int_{\Omega}Q(x)|u|^{2^*_s}dx-\frac{\lambda}{q+1}\int_{\Omega} |u|^{q+1}dx.
\end{equation*}
Plainly, positive critical points of $I_\lambda$ corresponds to solutions of \eqref{problem}.

In addition to the above definition, there is another useful equivalent characterization of the fractional Laplacian in terms of the so-called Dirichlet-to-Neumann operator (cf. \cite{bracoldepsan2013a,Cabre2010, cafsil2007an, Capella2011, Stinga2010}). 

Consider the cylinder $\mathcal C_\Omega:=\Omega\times(0,+\infty)\subset\R_+^{N+1}$ and denote by $\partial_L\mathcal{C}_\Omega:=\partial\Omega\times[0,+\infty)$ its lateral boundary. Set also $\SD^*:=\SD\times[0,+\infty)$,  $\SN^*:=\SN\times[0,+\infty)$ and $\Gamma^*:=\Gamma\times[0,+\infty)$. Note that, by construction, $\SD^*\cap\SN^*=\emptyset$, $\SD^*\cup\SN^*=\partial_L\mathcal C_\Omega$ and $\SD^*\cap\overline{\SN^*}=\Gamma^*$.

Given $u\in H_{\SD}^s(\Omega)$, we define its $s$-harmonic extension $w(x,y)=E_s[u(x)]$ as the solution to
\begin{equation*}
	\left\{
	\begin{array}{ll}\label{extprobl}
		-\text{div}\left(y^{1-2s}\nabla w(x,y) \right)=0 & \text{ in } \mathcal{C}_{\Omega}, \smallskip\\
		 B(w(x,y))=0  & \mbox{ on } \partial_L\mathcal{C}_{\Omega}, \smallskip\\
		 w(x,0)=u(x)  &  \text{ on } \Omega\times\{y=0\} ,
	\end{array}
	\right.
\end{equation*}
where 
\[B(w)=w\chi_{\SD^*} +\frac{\partial w}{\partial\nu}\chi_{\SN^*},\] 
being $\nu$, by a slight abuse of notation, the outwards unit normal to $\partial_L\mathcal C_\Omega$. The extension function $w$ belongs to the space
$$
\mathcal{X}_{\Sigma_{\mathcal{D}}^*}^s(\mathcal{C}_{\Omega})=\overline{C_0^\infty\left((\Omega\cup\SN)\times[0,+\infty)\right)}^{\left\| \cdot\right\|_{\mathcal{X}_{\Sigma_{\mathcal{D}}^*}^s(\mathcal{C}_{\Omega})}},
$$
where we define
\begin{equation}\label{norma}
	\left\| \cdot \right\|_{\mathcal{X}_{\Sigma_{\mathcal{D}}^*}^s(\mathcal{C}_{\Omega})}^2=k_s\int_{\mathcal C_\Omega} y^{1-2s}|\nabla (\cdot)|^2dxdy,
\end{equation}
with $k_s:=2^{2s-1}\frac{\Gamma(s)}{\Gamma(1-s)}$. The space $\mathcal{X}_{\Sigma_{\mathcal{D}}^*}^s(\mathcal{C}_{\Omega})$ is a Hilbert space equipped with the scalar product
$$
\left\langle w,z\right\rangle_{\mathcal{X}_{\Sigma_{\mathcal{D}}^*}^s(\mathcal{C}_{\Omega})}=k_s \int_{\mathcal C_\Omega}y^{1-2s}\nabla w\cdot\nabla z dx dy,\quad \text{for }\ w,z\in \mathcal{X}_{\Sigma_{\mathcal{D}}^*}^s(\mathcal{C}_{\Omega}),
$$
which induces the norm $\left\| \cdot \right\|_{\mathcal{X}_{\Sigma_{\mathcal{D}}^*}^s(\mathcal{C}_{\Omega})}$. Moreover the following inclusions hold
\begin{equation} \label{inclusions}
	\mathcal{X}_0^s(\mathcal{C}_{\Omega}) \subsetneq \mathcal{X}_{\Sigma_{\mathcal{D}}^*}^s(\mathcal{C}_{\Omega}) \subsetneq \mathcal{X}^s(\mathcal{C}_{\Omega}),
\end{equation}
being  $\mathcal{X}_0^s(\mathcal{C}_{\Omega})$ the space of functions that belong to $\mathcal{X}^s(\mathcal{C}_{\Omega})\equiv H^1(\mathcal{C}_{\Omega},y^{1-2s}dxdy)$ and vanish on $\partial_L\mathcal{C}_{\Omega}$. The key point of the extension technique is the  localization formula  (cf. \cite{bracoldepsan2013a,cafsil2007an}),
$$
\frac{\partial w}{\partial \nu^s}:=-k_s\lim_{y\to 0^+}y^{1-2s}\frac{\partial w}{\partial y}=(-\Delta)^s u(x),
$$
thanks to which problem \eqref{problem} can be restated as
\begin{equation}\label{extorigprobl}        \tag{$P_{\lambda,q}^*$}
	\left\{
	\begin{array}{ll}
		-\text{div}(y^{1-2s}\nabla w )=0 & \text{ in } \mathcal{C}_{\Omega}, \smallskip\\
		 B(w)=0 & \text{ on } \partial_L\mathcal{C}_{\Omega}, \smallskip\\
		 \displaystyle\frac{\partial w}{\partial \nu^s}= \lambda w(x,0)^q + Q(x)w(x,0)^{2^*_s-1} &  \text{ on } \Omega\times\{y=0\}.
	\end{array}
	\right.
\end{equation}
We say that $w\in \mathcal{X}_{\Sigma_{\mathcal{D}}^*}^s(\mathcal{C}_{\Omega})$ is a weak solution to \eqref{extorigprobl} if, for all $z\in \mathcal{X}_{\Sigma_{\mathcal{D}}^*}^s(\mathcal{C}_{\Omega})$,
		\begin{equation*}
			k_s\int_{\mathcal C_\Omega} y^{1-2s} \nabla w\cdot \nabla z dxdy =\int_\Omega Q(x)w(x,0)^{2^*_s-1}z(x,0)dx + \lambda\int_\Omega w(x,0)^{q}z(x,0)dx.
		\end{equation*}
Given $w\in \mathcal{X}_{\Sigma_{\mathcal{D}}^*}^s(\mathcal{C}_{\Omega})$ a weak solution to \eqref{extorigprobl}, its trace
$u(x)=\text{Tr}[w](x):=w(x,0)$ belongs to $H_{\Sigma_{\mathcal{D}}}^s(\Omega)$ and it is a weak solution to
problem \eqref{problem} and vice versa, if $u\in H_{\Sigma_{\mathcal{D}}}^s(\Omega)$ is a solution to \eqref{problem}, then $w=E_s[u]\in
\mathcal{X}_{\Sigma_{\mathcal{D}}^*}^s(\mathcal{C}_{\Omega})$ satisfies \eqref{extorigprobl}. Thus, both formulations are equivalent and the {\it extension operator}
$$
E_s: H_{\Sigma_{\mathcal{D}}}^s(\Omega) \to \mathcal{X}_{\Sigma_{\mathcal{D}}^*}^s(\mathcal{C}_{\Omega}),
$$
allows us to switch between each other. Moreover, due to the definition of the constant $k_s$ in \eqref{norma}, the extension operator $E_s$ is an isometry (cf. \cite{bracoldepsan2013a, cafsil2007an}),
\begin{equation}\label{norma2}
	\|E_s[u] \|_{\mathcal{X}_{\Sigma_{\mathcal{D}}^*}^s(\mathcal{C}_{\Omega})}=
	\|u\|_{H_{\Sigma_{\mathcal{D}}}^s(\Omega)}, \quad \text{for all } u\in H_{\Sigma_{\mathcal{D}}}^s(\Omega).
\end{equation}
Finally, the energy functional corresponding to \eqref{extorigprobl} is defined by
\begin{equation*}
	J_\lambda(w):=\frac{\kappa_s}{2}\int_{\mathcal{C}_{\Omega}}y^{1-2s}|\nabla w|^2dxdy-\frac{1}{2^*_s}\int_{\Omega}Q(x)|w(x,0)|^{2^*_s}dx-\frac{\lambda}{q+1}\int_{\Omega}|w(x,0)|^{q+1}dx,
\end{equation*}
so critical points of $J_\lambda$ in
$\mathcal{X}_{\Sigma_{\mathcal{D}}^*}^s(\mathcal{C}_{\Omega})$
correspond to critical points of $I_\lambda$ in
$H_{\Sigma_{\mathcal{D}}}^s(\Omega)$. Moreover, minima of $J_\lambda$ are also minima of $I_\lambda$ (the proof follows as in \cite[Proposition 3.1]{Barrios2012}).

When one considers Dirichlet boundary conditions the following \textit{trace inequality} holds (cf. \cite[Theorem 4.4]{bracoldepsan2013a}): there exists $C=C(N,s,r,|\Omega|)>0$ such that, for all $w\in\mathcal{X}_0^s(\mathcal{C}_{\Omega})$,
\begin{equation}\label{sobext}
	C\left(\int_{\Omega}|w(x,0)|^r dx\right)^{\frac{2}{r}}\leq k_s\int_{\mathcal{C}_{\Omega}}y^{1-2s}|\nabla w(x,y)|^2dxdy,
\end{equation}
for $r\in[1,2_s^*]$ and $N>2s$. Because of \eqref{norma2}, inequality \eqref{sobext} is equivalent to the fractional Sobolev inequality: for all $u\in H_{0}^s(\Omega)$ and $r\in[1,2^*_s]$, $N>2s$, one has
\begin{equation}\label{sobolev}
	C\left(\int_{\Omega}|u|^rdx\right)^{\frac{2}{r}}\leq \int_{\Omega}\left|(-\Delta)^{\frac{s}2}u\right|^2dx.
\end{equation}
If $r=2_s^*$ the best constant in \eqref{sobolev} (and, by \eqref{norma2}, also in \eqref{sobext}), 
denoted by $S(N,s)$, is independent of $\Omega$ and its exact value is given by
\begin{equation}\label{soboconst}
	S(s,N)=2^{2s}\pi^s\frac{\Gamma\left(\frac{N+2s}{2}\right)}{\Gamma\left(\frac{N-2s}{2}\right)}\left(\frac{\Gamma(\frac{N}{2})}{\Gamma(N)}\right)^{\frac{2s}{N}}.
\end{equation}
Since it is not achieved in any bounded domain, one has
\begin{equation*}
	S(s,N)\left(\int_{\mathbb{R}^N}|w(x,0)|^{\frac{2N}{N-2s}}dx\right)^{\frac{N-2s}{N}}\leq k_s\int_{\mathbb{R}_{+}^{N+1}}y^{1-2s}|\nabla w(x,y)|^2dxdy,
\end{equation*}
for all $w\in \mathcal{X}^s(\mathbb{R}_{+}^{N+1})$ where $\mathcal{X}^s(\mathbb{R}_{+}^{N+1})=\overline{\mathcal{C}^{\infty}(\mathbb{R}^N\times[0,\infty))}^{\|\cdot\|_{\mathcal{X}^s(\mathbb{R}_{+}^{N+1})}}$. In the whole space, the equality is achieved at the family $w_{\varepsilon}= E_s[u_{\varepsilon}]$,
\[
	u_{\varepsilon}(x)=\frac{\varepsilon^{\frac{N-2s}{2}}}{(\varepsilon^2+|x|^2)^{\frac{N-2s}{2}}},
\]
with arbitrary $\varepsilon>0$, (cf. \cite{bracoldepsan2013a}).

When mixed boundary conditions are considered the situation is quite similar since the Dirichlet condition is imposed on a set $\Sigma_{\mathcal{D}} \subset \partial \Omega$ with $|\Sigma_{\mathcal{D}}|=\mu$ and $0<\mu<|\partial\Omega|$.

\begin{definition}
	{\rm
		The Sobolev constant relative to the Dirichlet boundary $\Sigma_{\mathcal{D}}$ is defined by
		\begin{equation}\label{defi_sob}
			S(\Sigma_{\mathcal{D}})=\inf_{\substack{u\in
					H_{\Sigma_{\mathcal{D}}}^s(\Omega)\\ u\not\equiv
					0}}\frac{\|u\|_{H_{\Sigma_{\mathcal{D}}}^s(\Omega)}^2}{\|u\|_{L^{2_s^*}(\Omega)}^2}=\inf_{\substack{w\in \mathcal{X}_{\Sigma_{\mathcal{D}}^*}^s(\mathcal{C}_{\Omega})\\
					w\not\equiv
					0}}\frac{\|w\|_{\mathcal{X}_{\Sigma_{\mathcal{D}}}^s(\mathcal{C}_{\Omega})}^2}{\|w(\cdot,0)\|_{L^{2_s^*}(\Omega)}^2}.
		\end{equation}
	}
\end{definition}

Since $0<\mu<|\partial\Omega|$, by the inclusions \eqref{inclusions}, one has
\begin{equation}\label{const}
	0<S(\Sigma_{\mathcal{D}})<S(s,N).
\end{equation}
Here, the main difference with Dirichet boundary problems comes into play. Actually, because of \cite[Proposition 3.6]{colort2019the} it turns out that
\begin{equation}\label{relatsobconst}
S(\Sigma_{\mathcal{D}})\leq 2^{-\frac{2s}{N}}S(s,N).
\end{equation}
Thus we consider separately the hypotheses
\begin{equation}\label{c<}\tag{$C_<$}
S(\Sigma_{\mathcal{D}})<2^{-\frac{2s}{N}}S(s,N),
\end{equation}
and
\begin{equation}\label{c=}\tag{$C_=$}
S(\Sigma_{\mathcal{D}})=2^{-\frac{2s}{N}}S(s,N).
\end{equation}
Because of \cite[Theorem 2.9]{colort2019the}, under \eqref{c<} the constant $S(\Sigma_{\mathcal{D}})$ is attained while, in the lines of the Dirichlet case, under \eqref{c=} the constant $S(\Sigma_{\mathcal{D}})$ is not attained. By Hölder inequality, $S(\Sigma_{\mathcal{D}})\leq|\Omega|^{\frac{2s}{N}}\lambda_1^s(\mu)$, with $\lambda_1(\mu)$ the first eigenvalue of $(-\Delta)$ endowed with mixed boundary data on $\Sigma_{\mathcal{D}}=\Sigma_{\mathcal{D}}(\mu)$ and $\Sigma_{\mathcal{N}}= \Sigma_{\mathcal{N}}(\mu)$. Since $\lambda_1(\mu)\to 0$ as $\mu\to 0^+$, (cf. \cite[Lemma 4.3]{colper2003semilinear}), then $S(\Sigma_{\mathcal{D}})\to 0$ as $\mu\to 0^+$. Thus, taking the Dirichlet part small enough, the constant $S(\Sigma_{\mathcal{D}})$ is attained independently of the geometry of $\Omega$. This is in contrast to the Dirichlet case where, by a Pohozaev type identity, if $\Omega$ is a star-shaped domain the constant $S(s,N)$ is not attained. Nevertheless, geometric conditions on the boundary parts $\Sigma_{\mathcal{D}}$ and $\Sigma_{\mathcal{N}}$ can be also set in order to obtain a non-attainability result as the one for the Dirichlet case (cf. \cite[Corollary 5.2]{colort2019the}).

By \eqref{norma2} and \eqref{const}, the following trace inequality for the mixed boundary data setting holds.
\begin{lemma}\cite[Lemma 2.4]{colort2019the}\label{lem:traceineq}
	For all $w\in \mathcal{X}_{\Sigma_{\mathcal{D}}^*}^s(\mathcal{C}_{\Omega})$, one has
	\begin{equation*}
		S(\Sigma_{\mathcal{D}})\left(\int_\Omega|w(x,0)|^{2^*_s}  dx\right)^{\frac{2}{2^*_s}}\leq k_s\int_{\mathcal{C}_{\Omega}} y^{1-2s} |\nabla w(x,y)|^2 dxdy.
	\end{equation*}
	In an equivalent way, for all $u\in H_{\SD}^s(\Omega)$ it holds
	\begin{equation*}
		S(\Sigma_{\mathcal{D}})\left(\int_{\Omega}|u|^{2^*_s}dx\right)^{\frac{2}{2^*_s}}\leq \int_{\Omega}|(-\Delta)^{\frac{s}2}u|^2dx.
	\end{equation*}
\end{lemma}

\subsection{Notation}
Throughout this paper, the symbols $C,C_i$, $i=1,2,\ldots,$ will denote positive constants whose exact value may vary from line to line. Given $x\in\R^N$ and $r>0$, the symbols $B_r(x)$ and $\overline{B}_r(x)$ will stand for the open and the closed ball of $\R^N$ centered at $x$ and with radius $r$, respectively, while $\Omega_r(x):=\Omega\cap B_r(x)$. 

\section{Existence of solutions}\label{sec:existence}
We start by proving that we can recover compactness below a critical threshold. As usual, this is done by obtaining the range for which the Palais-Smale condition holds. 
\begin{definition}
Let $V$ be a Banach space. We say that $\{u_n\} \subset V$ is a $(PS)_c$ sequence for a functional $\mathfrak{F}:V\mapsto\mathbb{R}$ if
\begin{equation*}
\mathfrak{F}(u_n) \to c \quad \hbox{ and }\quad  \mathfrak{F}'(u_n) \to 0 \mbox{ in } V^* \quad \hbox{as } n\to + \infty,
\end{equation*}
where $V^*$ is the dual space of $V$. Moreover, we say that $\mathfrak{F}$ satisfies the $(PS)_c$ condition if every $(PS)_c$ sequence for $\mathfrak{F}$ has a strongly convergent subsequence.
\end{definition}

\begin{proposition}\label{PScondition}
Let 
\begin{equation}\label{cupperbound}
c<c^\star:=\displaystyle\frac{s}{N}\frac{S(\Sigma_D)^\frac{N}{2s}}{Q_M^\frac{N-2s}{2s}}.
\end{equation} 
Then $J_\lambda$ satisfies the $(PS)_c$ condition in the following cases:
\begin{itemize}
	\item[$(i)$] $q=1$ and $\lambda\in(0,\lambda_1^s)$;
	\item[$(ii)$]  $q\in(1,2^*_s-1)$ and $\lambda>0$.
\end{itemize}	
\end{proposition}

\begin{proof}
Let $c\in\R$ as in \eqref{cupperbound} and let $\{w_n\}\subset\X$ be a $(PS)_c$ sequence for $J_\lambda$, i.e.
$$
J_\lambda(w_n)=c+o(1), \quad J'_\lambda(w_n)=o(1) \quad \text{as } n\to+\infty.
$$
Let us show that $\{w_n\}$ is bounded in $\X$. Assuming $(i)$, for $n\in\N$ large enough we get
\begin{align*}
|c|+o(1)+\left\| w_n\right\|_{\X} & \geq J_\lambda(w_n) -\frac{1}{2^*_s}J'_\lambda(w_n)(w_n)\\
&= \left( \frac{1}{2}-\frac{1}{2^*_s}\right)\left\|w_n \right\|_{\X}^2-\lambda\left( \frac{1}{2}-\frac{1}{2^*_s}\right)\| w_n(x,0)\|_{L^2(\Omega)}^2\\
& \geq \frac{s}{N}\left( 1-\frac{\lambda}{\lambda_1^s}\right)\left\|w_n \right\|_{\X}^2, 
\end{align*} 
which proves the claim. If instead $(ii)$ holds, choosing $m\in(2,q+1)$ and again $n\in\N$ large enough, we obtain
\begin{align*}
	|c|+o(1)+\left\| w_n\right\|_{\X} & \geq J_\lambda(w_n) -\frac{1}{m}J'_\lambda(w_n)(w_n)\\
	&= \left( \frac{1}{2}-\frac{1}{m}\right)\left\|w_n \right\|_{\X}^2 +\left( \frac{1}{m}-\frac{1}{2^*_s}\right)\int_\Omega Q(x) |w_n(x,0)|^{2^*_s}dx \\
	&\;\;\;  +\lambda\left( \frac{1}{m}-\frac{1}{q+1}\right)\int_\Omega |w_n(x,0)|^{q+1}dx\\
	& \geq \left( \frac{1}{2}-\frac{1}{m}\right)\left\|w_n \right\|_{\X}^2, 
\end{align*} 
and the claim follows. From the boundedness of $\{w_n\}$ in $\X$ we deduce that, up to subsequences, $w_n\rightharpoonup w_0\in\X$ and, since the trace operator is bounded and linear, then $\Tr(w_n)\rightharpoonup \Tr(w_0)\in H_{\SD}^s(\Omega)$. By the compactness of the embedding $H_{\SD}^s(\Omega)\hookrightarrow L^p(\Omega)$, $p\in(1,2^*_s)$, we deduce then
\begin{equation}\label{convergences}
\begin{split}
	\Tr(w_n)\to\Tr(w_0) & \quad\text{ in } L^p(\Omega),\\
	\Tr(w_n)\to\Tr(w_0) & \quad\text{ a.e. in } \Omega.
\end{split}
\end{equation}
On the other hand, by Brezis-Lieb's Lemma we obtain
\begin{equation}\label{brezislieb}
	\begin{split}
		\left\| w_n-w_0\right\|_{\X}^2 =&\left\| w_n\right\|_{\X}^2 - \left\| w_0\right\|_{\X}^2 +o(1),\\
		\int_{\Omega}  Q(x)\left|w_n(x,0)-w_0(x,0)\right|^{2^*_s}dx  =& \int_{\Omega}  Q(x)\left|w_n(x,0)\right|^{2^*_s}dx \\
													    &- \int_{\Omega}  Q(x)\left|w_0(x,0)\right|^{2^*_s}dx +o(1),    
	\end{split}
\end{equation}
so we get
\begin{equation*}
\begin{split}	
	\left\|w_n-w_0 \right\|_{\X}^2 +\left\|w_0 \right\|_{\X}^2 &- \int_\Omega Q(x) |w_n(x,0)-w_0(x,0)|^{2^*_s}dx\\ 
	&- \int_\Omega Q(x)|w_0(x,0)|^{2^*_s}dx -\lambda\int_\Omega |w_0(x,0)|^{q+1}dx =o(1), 
\end{split}
\end{equation*}
and
\begin{equation}\label{limJprimelambda}
\lim_{n\to\infty}J'_\lambda(w_n)(w_0)=\left\| w_0\right\|_{\X}^2 - \int_\Omega Q(x)|w_0(x,0)|^{2^*_s}dx -\lambda\int_\Omega |w_0(x,0)|^{q+1}dx = 0.
\end{equation}
Collecting \eqref{convergences}, \eqref{brezislieb} and \eqref{limJprimelambda} we arrive at
\begin{equation}\label{stima1}
\frac{1}{2}\left\| w_n-w_0\right\|_{\X}^2 -\frac{1}{2^*_s} \int_\Omega Q(x)|w_n(x,0)-w_0(x,0)|^{2^*_s} dx =c-J_\lambda(w_0) +o(1) 
\end{equation}
and
\begin{equation*}
\left\| w_n-w_0\right\|_{\X}^2 -  \int_\Omega Q(x)|w_n(x,0)-w_0(x,0)|^{2^*_s} dx=o(1).
\end{equation*}
Let 
\begin{equation}\label{limitl}
l=\lim_{n\to\infty}\left\| w_n-w_0\right\|_{\X}^2=\lim_{n\to\infty}\int_\Omega Q(x)|w_n(x,0)-w_0(x,0)|^{2^*_s} dx.
\end{equation}
Taking also account of Sobolev inequality in Lemma \ref{lem:traceineq} we get
\begin{align*}
\left\| w_n-w_0\right\|_{\X}^2 &\geq S(\Sigma_D)
\left( \int_\Omega |w_n(x,0)-w_0(x,0)|^{2^*_s} dx\right)^\frac{2}{2^*_s} \\
&\geq  S(\Sigma_D)\left(\int_\Omega \frac{Q(x)}{Q_M}|w_n(x,0)-w_0(x,0)|^{2^*_s} dx\right)^\frac{2}{2^*_s},
\end{align*}
and passing to the limit for $n\to\infty$, we obtain
\[
l\geq S(\Sigma_D)\left(\frac{l}{Q_M}\right)^\frac{2}{2^*_s},
\]
which implies that either $l=0$ or $l\geq S(\Sigma_D)^\frac{N}{2s} Q_M^{-\frac{N-2s}{2s}}$. The second alternative can not occur. Otherwise, by \eqref{stima1} and \eqref{limitl} we would obtain
$$
c=\frac{l}{2}-\frac{l}{2^*_s} + J_\lambda(w_0)=\frac{s}{N}l+J_\lambda(w_0) \geq \frac{s}{N}\frac{S(\Sigma_D)^\frac{N}{2s}}{Q_M^\frac{N-2s}{2s}} +J_\lambda(w_0),
$$
which, together with \eqref{cupperbound}, leads to
$$
J_\lambda(w_0)\leq c - \frac{s}{N}\frac{S(\Sigma_D)^\frac{N}{2s}}{Q_M^\frac{N-2s}{2s}}<0.
$$
On the other hand, by \eqref{limJprimelambda}, 
$$
J_\lambda(w_0)=\left( \frac{1}{2}-\frac{1}{q+1}\right)\left\|w_0 \right\|_{\X}^2 + \left( \frac{1}{q+1}-\frac{1}{2^*_s}\right)\int_\Omega Q(x)|w_0|^{2^*_s}dx \geq 0,  
$$
a contradiction. As a result, it must be $l=0$, and hence $w_n\to w_0$ in $\X$ as $n\to\infty$. 
\end{proof}	

Now let $\phi_0\in C^\infty(\R^+)$ be a non-increasing cut-off function, $\phi_0(t)=1$ for $t\in[0,1/2]$, $\phi_0(t)=0$ for $t\geq 1$, and $|\phi'_0(t)|\leq C$ for all $t\geq 0$. Choose $\varrho>0$ small enough such that $B^+_\varrho((a^i,0)):=B_\varrho((a^i,0))\cap\R^{N+1}_+\subset \mathcal C_\Omega$ for every $i=0,1,\ldots,k$, and set 
\[z_{\varrho,\varepsilon}^i :=\phi_{\varrho,i}w_{\varepsilon,i}, \qquad\text{where}\qquad  w_{\varepsilon,i}:=E_s[u_{\varepsilon,i}],\]
and
\begin{align*}
\phi_{\varrho,i}(x,y)&:=\phi_0\left( \frac{|x-a^i,y|}{\varrho}\right)=\phi_0\left( \frac{(|x-a^i|^2+y^2)^\frac{1}{2}}{\varrho}\right), \quad (x,y)\in \mathcal C_\Omega,\\
u_{\varepsilon,i}(x)&:=	\frac{\varepsilon^{\frac{N-2s}{2}}}{(\varepsilon^2+|x-a^i|^2)^{\frac{N-2s}{2}}}, \quad x\in\R^N,\\
\end{align*}
for $i=1,\ldots,k$. The following relations hold (cf. \cite[Lemma 12]{ort2023concave}):
\begin{equation}\label{estimtruncextremals}
	\begin{split}
		\left\| z^i_{\varrho,\varepsilon}\right\|_{\X}^2 &= \frac{1}{2}\left\| w_{\varepsilon,i}\right\|_{\X}^2 + O\left( \left( \frac{\varepsilon}{\varrho}\right)^{N-2s} \right),\\
		\left\| \phi_{\varrho,i}(\cdot,0)u_{\varepsilon,i}\right\|_{2^*_s}^{2^*_s} & = \frac{1}{2}\left\| u_{\varepsilon,i}\right\|_{L^{2^*_s}(\R^N)}^{2^*_s} + O\left( \left( \frac{\varepsilon}{\varrho}\right)^{N} \right).  
	\end{split}
\end{equation}

Next, let us estimate the $L^p$-norm on $\Omega$ of the truncated functions $z^0_{\varrho,\varepsilon}(\cdot,0)$.
\begin{lemma}\label{Lpnormextremals}
The following facts hold:
\begin{itemize}
\item[$(i)$] if $\displaystyle p=\frac{N}{N-2s}$, then $C_1\varepsilon^\frac{N}{2}|\log\varepsilon| \leq\left\|z^0_{\varrho,\varepsilon}(\cdot,0)\right\|_{L^p(\Omega)}^p \leq C_2\varepsilon^\frac{N}{2}|\log\varepsilon|$;
\item[$(ii)$] if $\displaystyle p<\frac{N}{N-2s}$, then $C_3\varepsilon^\frac{(N-2s)p}{2}\leq\left\|z^0_{\varrho,\varepsilon}(\cdot,0)\right\|_{L^p(\Omega)}^p \leq C_4 \varepsilon^\frac{(N-2s)p}{2}$;
\item[$(iii)$] if $\displaystyle p>\frac{N}{N-2s}$, then $C_5\varepsilon^{N-\frac{(N-2s)p}{2}}\leq\left\|z^0_{\varrho,\varepsilon}(\cdot,0)\right\|_{L^p(\Omega)}^p\leq C_6 \varepsilon^{N-\frac{(N-2s)p}{2}}$,
\end{itemize}
where $C_i$, $i=1,\ldots,6$, are positive constants independent of $\varepsilon$.
\end{lemma}

\begin{proof}
Let $A_\varrho(a^0)$ be the half ball centered at $a^0$ and of radius $\varrho$ such that the outer unit normal vector $\nu$ to $\SN$ at $a^0$ is not contained in $A_\varrho(a^0)$. Since $a^0$ is a regular point, there exists a diffeomorphism $D_\varrho:\Omega_\varrho(a^0)\to A_\varrho(a^0)$ such that $D_\varrho(\partial\Omega_\varrho\cap\Omega)=\partial A_\varrho(a^0)\cap\partial B_\varrho(a^0)$  and $D_\varrho$ transforms $\partial\Omega_\varrho\cap\SN$ into the flat part of $\partial A_\varrho(a^0)$. Taking also this fact into account, we obtain
\begin{equation}\label{LPnormz0rhoepsilon}
	\begin{split}
\int_\Omega z^0_{\varrho,\varepsilon}(x,0)^p dx & = \int_{\Omega_\varrho(a^0)} z^0_{\varrho,\varepsilon}(x,0)^p dx\\ &=  \frac{\varepsilon^\frac{(N-2s)p}{2}}{2}\left(\int_{B_\varrho(a^0)\setminus B_{\frac{\varrho}{2}}(a^0)}\frac{1}{(\varepsilon^2+|x-a^0|^2)^\frac{(N-2s)p}{2}}dx\right. \\ 
&\;\;\;\left. + \int_{B_{\frac{\varrho}{2}}(a^0)}\frac{1}{(\varepsilon^2+|x-a^0|^2)^\frac{(N-2s)p}{2}}dx \right)  \\
&=O\left(\varepsilon^\frac{(N-2s)p}{2} \right) + \frac{\varepsilon^\frac{(N-2s)p}{2}}{2} \omega_N \int_0^\frac{\varrho}{2}\frac{r^{N-1}}{(\varepsilon^2+r^2)^\frac{(N-2s)p}{2}}dr\\ 
&= O\left(\varepsilon^\frac{(N-2s)p}{2} \right)+\frac{\varepsilon^{N-\frac{(N-2s)p}{2}}}{2}\omega_N \int_0^\frac{\varrho}{2\varepsilon} \frac{\sigma^{N-1}}{(1+\sigma^2)^\frac{(N-2s)p}{2}}d\sigma,
	\end{split}
	\end{equation}
where $\omega_N$ is the measure of the $(N-1)$-dimensional sphere $\mathbb{S}^{N-1}$ in $\R^N$. If $p=N/(N-2s)$, we can carry on the computations in \eqref{LPnormz0rhoepsilon} to get
\begin{align*}
\int_\Omega z^0_{\varrho,\varepsilon}(x,0)^p dx &= O\left(\varepsilon^\frac{N}{2} \right) +\frac{\varepsilon^\frac{N}{2}}{2}\omega_N\int_0^\frac{\varrho}{2\varepsilon} \frac{\sigma^{N-1}}{(1+\sigma^2)^\frac{N}{2}}d\sigma\\
& = O\left(\varepsilon^\frac{N}{2} \right) + \frac{\varepsilon^\frac{N}{2}}{2}\omega_N\int_0^\frac{\varrho}{2\varepsilon} \frac{1}{\sigma}d\sigma\\ 
&=O\left(\varepsilon^\frac{N}{2} \right) + O\left(\varepsilon^\frac{N}{2}|\log\varepsilon|\right)\\
&=O\left(\varepsilon^\frac{N}{2}|\log\varepsilon|\right), 
\end{align*}
which proves $(i)$. If $p<N/(N-2s)$, then $(N-2s)p-N+1<1$ and again from \eqref{LPnormz0rhoepsilon} we deduce
\begin{align*}
\int_\Omega z^0_{\varrho,\varepsilon}(x,0)^p dx & = O\left(\varepsilon^\frac{(N-2s)p}{2} \right) + \frac{\varepsilon^\frac{(N-2s)p}{2}}{2} \omega_N \int_0^\frac{\varrho}{2}\frac{r^{N-1}}{(\varepsilon^2+r^2)^\frac{(N-2s)p}{2}}dr\\
& \leq O\left(\varepsilon^\frac{(N-2s)p}{2} \right)+\frac{\varepsilon^{\frac{(N-2s)p}{2}}}{2}\omega_N\int_0^\frac{\varrho}{2}r^{-(N-2s)p+N-1}dr\\
&=O\left(\varepsilon^\frac{(N-2s)p}{2} \right) +\varepsilon^{\frac{(N-2s)p}{2}}O(1)= O\left(\varepsilon^\frac{(N-2s)p}{2} \right),
\end{align*}	
and hence $(ii)$ is satisfied. Finally, for $p>N/(N-2s)$, one has $(N-2s)p-N+1>1$ and so,
\begin{align*}
\int_\Omega z^0_{\varrho,\varepsilon}(x,0)^p dx & \leq O\left(\varepsilon^\frac{(N-2s)p}{2} \right) + \varepsilon^{N-\frac{(N-2s)p}{2}}O(1)\\
&\;\;\; + \frac{\varepsilon^{N-\frac{(N-2s)p}{2}}}{2}\omega_N\int_1^{+\infty} \sigma^{-(N-2s)p+N-1}d\sigma\\ 
& =O\left(\varepsilon^{N-\frac{(N-2s)p}{2}} \right), 
\end{align*}
which implies the validity of $(iii)$.
\end{proof}

\begin{lemma} \label{l5}
Assume $(Q_1)$. Then
	\begin{equation}\label{intQx}
		\left( \int_\Omega Q(x)z^0_{\varrho,\varepsilon}(x,0)^{2^*_s}dx\right)^\frac{2}{2^*_s}= \left( \frac{Q_M}{2}\right) ^\frac{2}{2^*_s}\left\| u_{\varepsilon,0}\right\|_{L^{2^*_s}}^{2^*_s}+o(\varepsilon^\alpha), \quad \text{as } \varepsilon\to 0^+.  
	\end{equation}
\end{lemma}

\begin{proof}
By $(Q_1)$, fixed $\eta>0$ there exists $\bar\varrho>0$ such that
	$$
	|Q(x)-Q(a^0)|<\eta|x-a^0|^\alpha, \quad \text{for all } x\in\Omega_{\bar\varrho}(a^0). 
	$$
	So, for sufficiently small $\varepsilon$ and $\bar\varrho>\sqrt{\varepsilon}$, we obtain
	\begin{align*}
		&\left| \int_\Omega Q(x)z^0_{\varrho,\varepsilon}(x,0)^{2^*_s}dx - \int_\Omega Q(a^0)z^0_{\varrho,\varepsilon}(x,0)^{2^*_s}dx\right| \\
		&\leq \frac{1}{2}\int_{\overline{B}_\varrho(a^0)} |Q(x)-Q(a^0)|z^0_{\varrho,\varepsilon}(x,0)^{2^*_s}dx\\
		& \leq \frac{\eta} {2}\varepsilon^N\int_{\overline{B}_{\bar\varrho}(a^0)} |x-a^0|^\alpha \frac{1}{\left(\varepsilon^2 + |x-a^0|^2 \right)^N }dx+\frac{\varepsilon^N}{2}\int_{\overline{B}_\varrho(a^0)\setminus B_{\bar\varrho}(a^0)} \frac{1}{\left(\varepsilon^2 + |x-a^0|^2 \right)^N }dx\\
		&=\frac{\eta\omega_N}{2}\varepsilon^N \int_0^{\bar\varrho}\frac{r^{\alpha+N-1}}{(\varepsilon^2+r^2)^N}dr + \frac{\omega_N}{2}\varepsilon^N \int_{\bar\varrho}^\varrho\frac{r^{N-1}}{(\varepsilon^2+r^2)^N}dr\\
		&=\frac{\eta\omega_N}{2}\varepsilon^\alpha \int_0^\frac{\bar\varrho}{\varepsilon}\frac{R^{\alpha+N-1}}{(1+R^2)^N}dR + \frac{\omega_N}{2} \int_\frac{\bar\varrho}{\varepsilon}^\frac{\varrho}{\varepsilon}\frac{R^{N-1}}{(1+R^2)^N}dR\\
		&\leq \eta C_1\varepsilon^\alpha + C_2\varepsilon^N,
	\end{align*}	
	for suitable constants $C_1,C_2>0$. Dividing through by $\varepsilon^\alpha$ and taking the $\limsup$ as $\varepsilon\to 0$, by the arbitrariness of $\eta$ we arrive at 
	$$
	\int_\Omega Q(x)z^0_{\varrho,\varepsilon}(x,0)^{2^*_s}dx= Q_M\int_\Omega  z^0_{\varrho,\varepsilon}(x,0)^{2^*_s}dx + o(\varepsilon^\alpha),
	$$
	and, by \eqref{estimtruncextremals},
	\begin{align*}
		\int_\Omega Q(x)z^0_{\varrho,\varepsilon}(x,0)^{2^*_s}dx & = \frac{Q_M}{2} \left\| u_{\varepsilon,0}\right\|^{2^*_s}_{L^{2^*}_s} + O\left( \left(\frac{\varepsilon}{\varrho}\right)^N\right) + o(\varepsilon^\alpha) = \frac{Q_M}{2} \left\| u_{\varepsilon,0}\right\|^{2^*_s}_{L^{2^*}_s}  + o(\varepsilon^\alpha).  
	\end{align*}
	Raising to the power $2/2^*_s$ both sides we get \eqref{intQx}.
\end{proof}

\begin{lemma}\label{propvarphilambda}
Let $q>1$. Then, for all $\lambda\geq 0$ and $\varepsilon>0$, the function $\varphi_{\lambda,\varepsilon}:[0,+\infty)\to\R$, given by
\begin{equation}\label{fibermap}
\varphi_{\lambda,\varepsilon}(t):=J_\lambda(t z^0_{\varrho,\varepsilon}), \quad\text{for all } t\geq 0,
\end{equation}
has a unique global maximizer $t_{\lambda,\varepsilon}>0$.  In particular, one has
\begin{equation}\label{t0epsilon}
t_{0,\varepsilon}=\left( \frac{\left\|z^0_{\varrho,\varepsilon} \right\|^2_{\X} }{\int_\Omega Q(x)z^0_{\varrho,\varepsilon}(x,0)^{2^*_s}dx}\right)^\frac{1}{2^*_s-2},
\end{equation}
and 
\begin{equation}\label{varphi0t0epsilon}
\varphi_{0,\varepsilon}(t_{0,\varepsilon})=\frac{s}{N}\frac{S(s,N)^\frac{N}{2s}}{2Q_M^\frac{N-2s}{2s}} + o(\varepsilon^\alpha)+ O\left(\left(\frac{\varepsilon}{\varrho}\right)^{N-2s}\right).
\end{equation}
Moreover, there exists $C_0,T_2>0$ constants, independent of both $\lambda$ and $\varepsilon$, such that
\begin{equation}\label{estimtlambdaepsq>1}
C_0(1+\lambda  \beta(\varepsilon))^{\frac{1}{1-q}} \leq t_{\lambda,\varepsilon} \leq T_2,
\end{equation}
for all $\lambda>0$ and for small $\varepsilon>0$ with $\beta(\varepsilon)$ given by
$$
\beta(\varepsilon):=\left\lbrace   
\begin{array}{ll}
\varepsilon^\frac{N}{2}|\log\varepsilon|  & \text{ if } \displaystyle q+1=\frac{N}{N-2s},\smallskip\\
\varepsilon^\frac{(N-2s)(q+1)}{2}  & \text{ if } \displaystyle q+1<\frac{N}{N-2s},\smallskip\\
\varepsilon^{N-\frac{(N-2s)(q+1)}{2}} & \text{ if } \displaystyle q+1>\frac{N}{N-2s}.
\end{array}
\right. 
$$	
If $q=1$, then $\varphi_{\lambda,\varepsilon}$ has a unique maximizer $t_{\lambda,\varepsilon}$ for every $\lambda\in(0,\lambda_1^s)$ and $\varepsilon>0$. Moreover, there exists a constant $T_1>0$, independent of both $\lambda$ and $\varepsilon$, such that
\begin{equation}\label{estimtlambdaepsq=1}
t_{\lambda,\varepsilon}\geq T_1,
\end{equation}
for every $\lambda\in(0,\lambda^s_1)$ and for small $\varepsilon>0$.
\end{lemma}

\begin{proof}
Assume first $q>1$. By \eqref{fibermap}, we deduce $\varphi_{\lambda,\varepsilon}'(t)=tg_{\lambda,\varepsilon}(t)$ for every $\lambda,\varepsilon>0$, where
\begin{equation}\label{derivativevarphi}	
g_{\lambda,\varepsilon}(t):=\kappa_s \int_{\mathcal C_{\Omega}}y^{1-2s}|\nabla z^0_{\varrho,\varepsilon}|^2 dxdy- t^{2^*_s-2}\!\!\int_\Omega Q(x)z^0_{\varrho,\varepsilon}(x,0)^{2^*_s}dx
-\lambda t^{q-1}\! \int_\Omega z^0_{\varrho,\varepsilon}(x,0)^{q+1}dx.
\end{equation}
Clearly $g_{\lambda,\varepsilon}(0)>0$, $g_{\lambda,\varepsilon}(t)\rightarrow -\infty$  as $t\to +\infty$ and $g_{\lambda,\varepsilon}'(t)<0$ for all $t\in (0,+\infty)$. Thus, there exists a unique $t_{\lambda,\varepsilon}\in (0,+\infty)$ such that $g_{\lambda,\varepsilon}(t)>0$ in $(0,t_{\lambda,\varepsilon})$ and  $g_{\lambda,\varepsilon}(t)<0$ in $t\in (t_{\lambda,\varepsilon},+\infty)$, that is, $t_{\lambda,\varepsilon}$ is the unique   
maximizer  of $\varphi_{\lambda,\varepsilon}$ in $(0,+\infty)$. Next, let us show that 
$$
T_\epsilon :=\frac{ \kappa_s \displaystyle{\int_{\mathcal C_{\Omega}}y^{1-2s}|\nabla z^0_{\varrho,\varepsilon}|^2 dxdy}}{\displaystyle{\int_\Omega Q(x)z^0_{\varrho,\varepsilon}(x,0)^{2^*_s}dx} },
$$
is uniformly upper bounded with respect to  small $\varepsilon.$ Actually, from \eqref{estimtruncextremals} and \eqref{intQx}, we have 
\begin{align*}
	T_\epsilon & =\frac{\displaystyle\frac{1}{2}\left\| w_{\varepsilon,0}\right\|_{\X}^2 + O\left( \left( \frac{\varepsilon}{\varrho}\right)^{N-2s} \right)}{\displaystyle  \frac{Q_M}{2} \left\| u_{\varepsilon,0}\right\|_{L^{2^*_s}}^{2^*_s}+o(\varepsilon^\alpha)} 
	\leq  \frac{\displaystyle\frac{1}{2}\|w_{\varepsilon,0}\|^2_{\mathcal{X}^s\left(\R_+^{N+1}\right) }  + O\left( \left( \frac{\varepsilon}{\varrho}\right)^{N-2s} \right)}{\displaystyle  \frac{Q_M}{2} \left\| u_{\varepsilon,0}\right\|_{L^{2^*_s}}^{2^*_s}+o(\varepsilon^\alpha)} \\
	&\leq \frac{S(s,N)}{Q_M}  + O\left(  \left( \frac{\varepsilon}{\varrho}\right)^{N-2s}\right). 
\end{align*}
Thus, since $g_{\lambda,\varepsilon}(t_{\lambda,\varepsilon})=0$, there exists some $T_2$ independent of both $\lambda$ and small $\varepsilon$ such that
 $$
 t_{\lambda,\varepsilon}\leq T_\epsilon^{\frac{1}{2^*_s-2}}\leq T_2.
 $$  
 On the other hand,  by \eqref{estimtruncextremals},  Lemmas  \ref{Lpnormextremals} and  \ref{l5}, we get
\begin{align*}
	t_{\lambda,\varepsilon}^{q-1} &=\frac{ \kappa_s \displaystyle{\int_{\mathcal C_{\Omega}}y^{1-2s}|\nabla z^0_{\varrho,\varepsilon}|^2 dxdy}} { t_{\lambda,\varepsilon}^{2^*_s-q-1}\displaystyle{\int_\Omega Q(x)z^0_{\varrho,\varepsilon}(x,0)^{2^*_s}dx+\lambda  \int_\Omega z^0_{\varrho,\varepsilon}(x,0)^{q+1}dx}}\\
	&\geq \frac{ \kappa_s \displaystyle{\int_{\mathcal C_{\Omega}}y^{1-2s}|\nabla z^0_{\varrho,\varepsilon}|^2 dxdy}} { T_{2}^{2^*_s-q-1}\displaystyle{\int_\Omega Q(x)z^0_{\varrho,\varepsilon}(x,0)^{2^*_s}dx+\lambda  \int_\Omega z^0_{\varrho,\varepsilon}(x,0)^{q+1}dx}} \\
&\geq \frac{C\displaystyle{\left(\int_\Omega z^0_{\varrho,\varepsilon}(x,0)^{2^*_s}dx\right)^{\frac{2}{2_s^*}}}   } { T_{2}^{2^*_s-q-1}\displaystyle{\int_\Omega Q(x)z^0_{\varrho,\varepsilon}(x,0)^{2^*_s}dx+\lambda  \int_\Omega z^0_{\varrho,\varepsilon}(x,0)^{q+1}dx}} \\
&\geq \frac{  \left\| u_{1,0}\right\|_{L^{2^*_s}(\R^N)}+o(\varepsilon) } { T_{2}^{2^*_s-q-1} \displaystyle{\left(\frac{Q_M}{2} \left\| u_{1,0}\right\|_{L^{2^*_s}}^{2^*_s(\R^N)}+o(\varepsilon^\alpha)\right)+C\lambda  \beta(\varepsilon)}}
\end{align*}
which yields 
$$
t_{\lambda,\varepsilon} \geq C_0(1+\lambda  \beta(\varepsilon))^{\frac{1}{1-q}} 
$$ 
for every $\lambda>0$, for small $\varepsilon$ and for some positive constant $C_0>0$. This proves \eqref{estimtlambdaepsq>1}. 

If $q=1$, for any $\lambda\in (0,\lambda_1^s)$ and small $\varepsilon$,  we get 
\begin{align*}
	t_{\lambda,\varepsilon}^{2_s^*-2}&= \frac{ \kappa_s \displaystyle{\int_{\mathcal C_{\Omega}}y^{1-2s}|\nabla z^0_{\varrho,\varepsilon}|^2 dxdy-\lambda  \int_\Omega z^0_{\varrho,\varepsilon}(x,0)^{2}dx}} { \displaystyle{\int_\Omega Q(x)z^0_{\varrho,\varepsilon}(x,0)^{2^*_s}dx}}\\
&\geq  \frac{  \displaystyle{ \left\| u_{1,0}\right\|_{L^{2^*_s}(\R^N)}+o(\varepsilon)-\lambda   \int_\Omega z^0_{\varrho,\varepsilon}(x,0)^{2}dx}} { \displaystyle{\int_\Omega Q(x)z^0_{\varrho,\varepsilon}(x,0)^{2^*_s}dx}}\\
& \geq C_0(1-C\lambda_1^s \beta(\varepsilon)), 
\end{align*} 
which gives \eqref{estimtlambdaepsq=1}. 
The expression in \eqref{t0epsilon} comes out by picking $\lambda=0$ in \eqref{derivativevarphi} and solving the equation $\varphi'_{0,\varepsilon}(t)=0$ for $t>0$. Finally, by using also \eqref{estimtruncextremals} and \eqref{intQx}, we have
 
\begin{align*}
\varphi_0(t_{0,\varepsilon}) & =J_0(t_{0,\varepsilon}z^0_{\varrho,\varepsilon}) =\frac{s}{N} \left(\frac{\left\| z^0_{\varrho,\varepsilon}\right\|_{\X}^2 }{\int_\Omega Q(x)z^0_{\varrho,\varepsilon}(x,0)^{2^*_s} dx} \right)^\frac{2}{2^*_s-2}\left\| z^0_{\varrho,\varepsilon}\right\|_{\X}^2\\
&=\frac{s}{N}\left( \frac{\left\| z^0_{\varrho,\varepsilon}\right\|_{\X}^2}{\left( \int_\Omega Q(x)z^0_{\varrho,\varepsilon}(x,0)^{2^*_s} dx\right)^\frac{2}{2^*_s}}\right)^\frac{N}{2s}\\
&=\frac{s}{N}\left( \frac{\frac{1}{2}\left\| w_{\varepsilon,0}\right\|_{\X}^2+ O\left( \left( \frac{\varepsilon}{\varrho}\right)^{N-2s}\right)  }{\left( \frac{Q_M}{2}\right) ^\frac{2}{2^*_s}\left\| u_{\varepsilon,0}\right\|^2_{2^*_s} +o(\varepsilon^\alpha)}\right)^\frac{N}{2s}\\
& =\frac{s}{N}\left( \frac{\frac{1}{2}S(s,N)^\frac{N}{2s}+ O\left( \left( \frac{\varepsilon}{\varrho}\right)^{N-2s}\right)  }{\left( \frac{Q_M}{2}\right) ^\frac{2}{2^*_s}S(s,N)^\frac{N-2s}{2s} +o(\varepsilon^\alpha)}\right)^\frac{N}{2s}\\  
&=\frac{s}{N}\left[\left( \frac{1}{2}\right)^\frac{N}{2s} S(s,N)^{\left( \frac{N}{2s}\right)^2}\left(1+ O\left( \left( \frac{\varepsilon}{\varrho}\right)^{N-2s}\right)\right)\left( \frac{Q_M}{2}\right) ^{-\frac{2}{2^*_s}\frac{N}{2s}}S(s,N)^{-\frac{N-2s}{2s}\frac{N}{2s}}(1+o(\varepsilon^\alpha))  \right] \\
&=\frac{s}{N}\frac{S(s,N)^\frac{N}{2s}}{2Q_M^\frac{N-2s}{2s}}\left( 1 + o(\varepsilon^\alpha)+ O\left(\left(\frac{\varepsilon}{\varrho}\right)^{N-2s}\right)\right),
\end{align*}
and \eqref{varphi0t0epsilon} is proved as well. The proof is then concluded.
\end{proof}

\begin{proposition}\label{supfiberingmap}
Under assumption $(Q_1)$ one has
\[
\sup_{t\geq 0}J_\lambda(tz^0_{\varrho,\varepsilon}) < \frac{s}{N}\frac{S(s,N)^\frac{N}{2s}}{2 Q_M^\frac{N-2s}{2s}},
\]
\begin{itemize}
\item for sufficiently large $\lambda$ in the case $(Q_1)$-$(i)$;
\item for every $\lambda>0$ in the cases $(Q_1)$-$(ii)$ and $(Q_1)$-$(iii)$, $q\neq 1$;
\item  for $\lambda\in(0,\lambda_1^s)$ in the case $(Q_1)$-$(iii)$, $q=1$. 
\end{itemize}
\end{proposition}

\begin{proof}
Notice at first that one has
\begin{align*}
\sup_{t\geq 0}	\varphi_{\lambda,\varepsilon}(t) &\leq \varphi_{\lambda,\varepsilon}(t_{\lambda,\varepsilon}) = \varphi_0 (t_{\lambda,\varepsilon}) - \frac{\lambda}{q+1}t_{\lambda,\varepsilon}^{q+1}\int_\Omega z^0_{\varrho,\varepsilon}(x,0)^{q+1}dx \\
& \leq \varphi_0 (t_{0,\varepsilon}) - \frac{\lambda}{q+1}t_{\lambda,\varepsilon}^{q+1}\int_\Omega z^0_{\varrho,\varepsilon}(x,0)^{q+1}dx\\
&=\frac{s}{N}\frac{S(s,N)^\frac{N}{2s}}{2 Q_M^\frac{N-2s}{2s}}\left( 1 + o(\varepsilon^\alpha)+ O\left(\left(\frac{\varepsilon}{\varrho}\right)^{N-2s}\right)\right)- \frac{\lambda}{q+1}t_{\lambda,\varepsilon}^{q+1}\int_\Omega z^0_{\varrho,\varepsilon}(x,0)^{q+1}dx.
\end{align*}
Now consider the three alternative cases in $(Q_1)$.

$(i)$ If $N=2$ or $N=3$ and $\displaystyle q+1\in\left(\left. 2,\frac{4s}{N-2s}\right]\right.$, 
we claim that
\begin{align*}
\sup_{t\geq 0}\varphi_{\lambda,\varepsilon}(t) & \leq
\left\lbrace 
\begin{array}{ll}
\displaystyle\frac{s}{N}\frac{S(s,N)^\frac{N}{2s}}{2 Q_M^\frac{N-2s}{2s}}+ O\left(\left(\frac{\varepsilon}{\varrho}\right)^{N-2s}\right)- C_1,
& \text{ if } \displaystyle 2<q+1<\frac{N}{N-2s} \smallskip \\ 
\displaystyle\frac{s}{N}\frac{S(s,N)^\frac{N}{2s}}{2 Q_M^\frac{N-2s}{2s}}+ O\left(\left(\frac{\varepsilon}{\varrho}\right)^{N-2s}\right)- C_2,
& \text{ if } \displaystyle q+1=\frac{N}{N-2s} \smallskip\\
\displaystyle\frac{s}{N}\frac{S(s,N)^\frac{N}{2s}}{2 Q_M^\frac{N-2s}{2s}}+ O\left(\left(\frac{\varepsilon}{\varrho}\right)^{N-2s}\right)- C_3,
& \text{ if } \displaystyle\frac{N}{N-2s}<q+1\leq\frac{4s}{N-2s}\\
\end{array}
\right.\\
&< \frac{s}{N}\frac{S(s,N)^\frac{N}{2s}}{2 Q_M^\frac{N-2s}{2s}},	
\end{align*}
for sufficiently small $\varepsilon>0$ and sufficiently large $\lambda>0$. Let us justify the inequality for $q+1\in\left(2,\frac{N}{N-2s} \right)$, the other two being similar. In this case $\beta(\varepsilon)=\varepsilon^\frac{(N-2s)(q+1)}{2}$ so letting 
$$
\lambda = \varepsilon^{-\frac{(N-2s)(q+1)}{2}},
$$ 
and taking also account of \eqref{estimtlambdaepsq>1} and Lemma \ref{Lpnormextremals} we get
\begin{align*}
\frac{\lambda}{q+1}t_{\lambda,\varepsilon}^{q+1}\int_\Omega z^0_{\varrho,\varepsilon}(x,0)^{q+1}dx & \geq \frac{\lambda C_0 C_3}{(q+1)(1+\lambda\beta(\varepsilon))^\frac{q+1}{q-1}}\varepsilon^\frac{(N-2s)(q+1)}{2}=\frac{C_0 C_3}{(q+1) 2^\frac{q+1}{q-1}}= C_1,
\end{align*}
as desired. For the other two inequalities, it suffices to choose $$
\lambda=\varepsilon^{-\frac{N}{2}}|\log(\varepsilon)|^{-1} \quad\text{and}\quad \lambda=\varepsilon^{-N+\frac{(N-2s)(q+1)}{2}},
$$
respectively, to get the conclusion.

$(ii)$ If $N=2$ or $N=3$ and $\displaystyle q+1\in\left(\frac{4s}{N-2s},\frac{2N}{N-2s}\right)$ we get
\begin{align*}
\sup_{t\geq 0}\varphi_{\lambda,\varepsilon}(t) &\leq \frac{s}{N}\frac{S(s,N)^\frac{N}{2s}}{2 Q_M^\frac{N-2s}{2s}} + o\left( \varepsilon^{N-\frac{(N-2s)(q+1)}{2}}\right)-C\frac{\lambda}{q+1}(1+\lambda   \beta(\varepsilon))^{\frac{q+1}{1-q}}\int_\Omega z^0_{\varrho,\varepsilon}(x,0)^{q+1}dx\\
&\leq  \frac{s}{N}\frac{S(s,N)^\frac{N}{2s}}{2 Q_M^\frac{N-2s}{2s}} + o\left( \varepsilon^{N-\frac{(N-2s)(q+1)}{2}}\right) -C_\lambda\varepsilon^{N-\frac{(N-2s)(q+1)}{2}}\\
&<\frac{s}{N}\frac{S(s,N)^\frac{N}{2s}}{2 Q_M^\frac{N-2s}{2s}},
\end{align*}
for small $\varepsilon$ and for all $\lambda>0$, where $C_\lambda$ is a positive constant depending on $\lambda$. 

$(iii)$ For $N\geq 4$ and $\displaystyle q+1\in \left[\left. 2,\frac{2N}{N-2s}\right)\right. $ we distinguish two cases. If $q+1\in \left(2,\frac{2N}{N-2s}\right)$, arguing as in the previous case we deduce
\begin{align*}
	\sup_{t\geq 0}\varphi_{\lambda,\varepsilon}(t) &\leq  \frac{s}{N}\frac{S(s,N)^\frac{N}{2s}}{2 Q_M^\frac{N-2s}{2s}} + o\left( \varepsilon^{N-\frac{(N-2s)(q+1)}{2}}\right) -C_\lambda\varepsilon^{N-\frac{(N-2s)(q+1)}{2}}<\frac{s}{N}\frac{S(s,N)^\frac{N}{2s}}{2 Q_M^\frac{N-2s}{2s}},
\end{align*}
again for small $\varepsilon$ and every $\lambda>0$. 
When $q=1$, by using also \eqref{estimtlambdaepsq=1} we deduce
\begin{align*}
	\sup_{t\geq 0}\varphi_{\lambda,\varepsilon}(t) &\leq  \frac{s}{N}\frac{S(s,N)^\frac{N}{2s}}{2 Q_M^\frac{N-2s}{2s}} + o\left( \varepsilon^{N-\frac{(N-2s)(q+1)}{2}}\right) -C\frac{\lambda}{q+1}T_1^{q+1}\varepsilon^{N-\frac{(N-2s)(q+1)}{2}}\\
	&\leq  \frac{s}{N}\frac{S(s,N)^\frac{N}{2s}}{2 Q_M^\frac{N-2s}{2s}} + o\left( \varepsilon^{N-\frac{(N-2s)(q+1)}{2}}\right) -C_\lambda\varepsilon^{N-\frac{(N-2s)(q+1)}{2}}\\
	&<\frac{s}{N}\frac{S(s,N)^\frac{N}{2s}}{2 Q_M^\frac{N-2s}{2s}},
\end{align*}
for small $\varepsilon$ and for $\lambda\in (0,\lambda_1^s)$.
\end{proof}		

We are now in a position to prove the main results of this section.

\begin{theorem}\label{thmexistence}
Assume \eqref{c=}. Then, the following holds.
\begin{itemize}
\item If $(Q_1)$-$(i)$ is satisfied then there exists $\lambda^\star>0$ such that \eqref{problem} has a positive solution for all $\lambda\geq\lambda^\star$.
\item If $(Q_1)$-$(ii)$ is satisfied, \eqref{problem} admits a positive solution for any $\lambda>0$.
\item If $(Q_1)$-$(iii)$ holds, \eqref{problem} has a solution for all $\lambda\in(0,\lambda_1^s)$ if $q=1$; for all $\lambda>0$ if $q\in(1,\frac{N+2s}{N-2s})$.
\end{itemize}
\end{theorem}

\begin{proof}
For any $w\in\X$ one has
$$
\int_\Omega \frac{Q(x)}{Q_M}|w(x,0)|^{2^*_s}dx \leq \int_\Omega|w(x,0)|^{2^*_s}dx
$$
and hence, by Sobolev embeddings,
$$
\int_\Omega Q(x)|w(x,0)|^{2^*_s} dx \leq Q_M S(s,N)^{-\frac{2^*_s}{2}}\left\| w\right\|_{\X}^{2^*_s}. 
$$
Then we get
\begin{align*}	
J_\lambda(w)&:=\frac{\kappa_s}{2}\int_{\mathcal{C}_{\Omega}}y^{1-2s}|\nabla w|^2dxdy-\frac{1}{2^*_s}\int_{\Omega}Q(x)|w(x,0)|^{2^*_s}dx-\frac{\lambda}{q+1}\int_{\Omega}|w(x,0)|^{q+1}dx\\
& \geq \frac{1}{2}\left\| w\right\|_{\X}^2 - \frac{Q_M}{2^*_s S(s,N)^\frac{2^*_s}{2}} \left\| w\right\|_{\X}^{2^*_s} - \frac{\lambda}{q+1}\int_{\Omega}|w(x,0)|^{q+1}dx\\ 
& = 
\left\lbrace 
\begin{array}{ll}
\displaystyle\frac{1}{2}\left\| w\right\|_{\X}^2 + o\left(\left\|w \right\|_{\X}^2 \right) & \text { for all } \lambda>0 \text{ if } q\in(1,2^*_s-1),\smallskip\\
\displaystyle\frac{1}{2}\left( 1-\frac{\lambda}{\lambda_1^s}\right) \left\| w\right\|_{\X}^2 + o\left(\left\|w \right\|_{\X}^2 \right) & \text { for all } \lambda\in(0,\lambda_1^s) \text{ if } q=1.\\ 
\end{array}
\right. 
\end{align*}
In both cases, $w=0$ turns out to be a local minimum of $J_\lambda$ in $\X$.

Since 
$\varphi_{\lambda,\varepsilon}(t)= J_\lambda(t z^0_{\varrho,\varepsilon})\to -\infty$ as $t\to +\infty$ for all $\lambda>0$ and for all $q\in[1,2^*_s-1)$, there exists $\bar t>0$ such that $\varphi_{\lambda,\varepsilon}(\bar t)<0$. Next, let us take $\bar w:=\bar t z^0_{\varrho,\varepsilon}$ and
\[c :=\inf_{\gamma\in\Gamma}\max_{t\in[0,1]}J_\lambda(\gamma(t)),\]
where $\Gamma :=\left\lbrace \gamma\in C^0([0,1],\X): \gamma(0)=0, \ \gamma(1)=\bar w\right\rbrace$. Because of Proposition \ref{supfiberingmap},
$$
c<\frac{s}{N}\frac{S(s,N)^\frac{N}{2s}}{2 Q_M^\frac{N-2s}{2s}}=\frac{s}{N}\frac{S(\SD)^\frac{N}{2s}}{Q_M^\frac{N-2s}{2s}}=c^\star,
$$
for sufficiently large $\lambda$ in the case $(Q_1)$-$(i)$; for every $\lambda>0$ if either $(Q_1)$-$(ii)$ or $(Q_1)$-$(iii)$ with $q\neq 1$ holds; for every $\lambda\in(0,\lambda_1^s)$ under $(Q_1)$-$(iii)$ with $q=1$. So, taking also account of Proposition \ref{cupperbound}, by the mountain pass theorem $c$ is a critical value for $J_\lambda$, i.e., there exists $w_0\in\X\setminus\{0\}$ such that $J_\lambda(w_0)=c$ and $J'_\lambda(w_0)=0$, in the following cases:
\begin{itemize}
	\item if $(Q_1)$-$(i)$ is satisfied, for sufficiently large $\lambda>0$;
	\item if $(Q_1)$-$(ii)$ is satisfied, for any $\lambda>0$;
	\item if $(Q_1)$-$(iii)$ is satisfied, for all $\lambda\in(0,\lambda_1^s)$ if $q=1$; for all $\lambda>0$ if $q\in(1,\frac{N+2s}{N-2s})$. 
\end{itemize}
Finally, since $J_\lambda(w)=J_\lambda(|w|)$, $w_0$ is non-negative, due to the strong maximum principle, $w_0>0$ a.e. in $\mathcal{C}_\Omega$. This completes the proof.
\end{proof}

In the case of strict inequality in \eqref{relatsobconst}, we get a result similar to Theorem \ref{thmexistence}, but the linear case $q=1$ is ruled out.

\begin{theorem}\label{thmexistence2}
Assume \eqref{c<}. Then, the following holds.
\begin{itemize}
\item If $(Q_1)$-$(i)$ is satisfied then there exists $\lambda^\star>0$ such that \eqref{problem} has a positive solution for all $\lambda\geq\lambda^\star$. 
\item If $(Q_1)$-$(ii)$ is satisfied, \eqref{problem} admits a positive solution for any $\lambda>0$. 
\item If $(Q_1)$-$(iii)$ holds, \eqref{problem} has a solution for all  $\lambda>0$ provided that $q\in(1,\frac{N+2s}{N-2s})$.
\end{itemize}
\end{theorem}

\begin{proof}
By \cite[Theorem 1.1--(iii)]{colort2019the}, the constant $S(\Sigma_{\mathcal{D}})$ is attained at some positive function $\tilde{u}\in H_{\Sigma_{\mathcal{D}}}^s(\Omega)$, namely, 
\[
\|\tilde{u}\|_{H_{\Sigma_{\mathcal{D}}}^s(\Omega)}^2=S(\Sigma_{\mathcal{D}}) \|\tilde{u}\|_{L^{2_s^*}(\Omega)}^2,
\]
so the critical problem 
\begin{equation*}
	\left\{
	\begin{array}{ll} 
(-\Delta)^s u = u^{2^*_s-1} & \text{ in } \Omega,\smallskip\\
u>0 & \text{ in } \Omega,\smallskip\\
 B(u)=0 & \text{ on } \partial\Omega
	\end{array}
	\right.
\end{equation*}
has a solution $u=[S(\Sigma_{\mathcal{D}})]^{\frac{1}{2_s^*-2}}\tilde{u}$. Without loss of generality let us take $\tilde{u}$ such that $\|\tilde{u}\|_{L^{2_s^*}(\Omega)}^2=1$ and denote by $\tilde{w}=E_s[\tilde{u}]$ its $s$-harmonic extension, so that
\[\|\tilde{w}\|_{\mathcal{X}_{\Sigma_{\mathcal{D}}}^s(\mathcal{C}_{\Omega})}^2=S(\Sigma_{\mathcal{D}}).\]
As before, for every $\lambda\geq 0$, let us consider the fibering map
$$
\psi_\lambda(t):=J_{\lambda,Q_M}(t \tilde{w}), \quad\text{for all } t\geq 0,
$$ 
where 
\begin{equation*}
	J_{\lambda,Q_M}(w):=\frac{\kappa_s}{2}\int_{\mathcal{C}_{\Omega}}y^{1-2s}|\nabla w|^2dxdy-\frac{1}{2^*_s}\int_{\Omega}Q_M |w(x,0)|^{2^*_s}dx-\frac{\lambda}{q+1}\int_{\Omega}|w(x,0)|^{q+1}dx.
\end{equation*}

By the same arguments as Lemma \ref{propvarphilambda}, 
$\psi_\lambda$ has a unique maximizer $t_{\lambda}>0$ and $t_{\lambda}\leq T_2$ for some $T_2>0$ independent of $\lambda$. 
In particular,
\begin{equation*}
 t_{0}=\left( \frac{S(\Sigma_{\mathcal{D}}) }{\displaystyle\int_\Omega Q_M|\tilde{w}(x,0)|^{2^*_s}dx}\right)^\frac{1}{2^*_s-2}
 \qquad\text{and}\qquad
\psi_0(t_0)= \frac{s}{N}\left(\frac{S(\Sigma_{\mathcal{D}})}{ Q_M^{\frac{2}{2_s^*}}}\right)^{\frac{N}{2s}}. 
\end{equation*}
Now, let
\begin{equation*}
	 \zeta_\lambda(t):= J_{\lambda}(t \tilde{w}), \quad\text{for all } t\geq 0.
\end{equation*}
One has
\begin{align*}
\sup_{t\geq 0}	\zeta_\lambda(t)  & =\sup_{t\geq 0} \left[	\psi_\lambda(t)+ \frac{1}{2^*_s}t^{2^*_s}\int_{\Omega}(Q_M-Q(x))|\tilde{w}(x,0)|^{2^*_s}dx\right]\\ 
 &\leq \psi_0 (t_{0}) +\frac{1}{2^*_s}t_{\lambda}^{2^*_s}\int_{\Omega}(Q_M-Q(x))|\tilde{w}(x,0)|^{2^*_s}dx- \frac{\lambda}{q+1}t_{\lambda}^{q+1}\int_\Omega |\tilde{w}(x,0)|^{q+1}dx\\
 & =\frac{s}{N}\left(\frac{S(\Sigma_{\mathcal{D}})}{ Q_M^{\frac{2}{2_s^*}}}\right)^{\frac{N}{2s}}+\frac{1}{2^*_s}t_{\lambda}^{2^*_s}\int_{\Omega}(Q_M-Q(x))|\tilde{w}(x,0)|^{2^*_s}dx- \frac{\lambda}{q+1}t_{\lambda}^{q+1}\int_\Omega |\tilde{w}(x,0)|^{q+1}dx\\
&=\frac{s}{N}\left(\frac{S(\Sigma_{\mathcal{D}})}{ Q_M^{\frac{2}{2_s^*}}}\right)^{\frac{N}{2s}}+t_{\lambda}^{q-1}\left[\frac{1}{2^*_s}t_{\lambda}^{2^*_s-q+1}\int_{\Omega}(Q_M-Q(x))|\tilde{w}(x,0)|^{2^*_s}dx\right. \\ 
&\;\;\;\left. -\frac{\lambda}{q+1} \int_\Omega |\tilde{w}(x,0)|^{q+1}dx\right].
\end{align*}
If 
$$
\frac{1}{2^*_s}t_{\lambda}^{2^*_s-q+1}\int_{\Omega}(Q_M-Q(x))|\tilde{w}(x,0)|^{2^*_s}dx- \frac{\lambda}{q+1} \int_\Omega |\tilde{w}(x,0)|^{q+1}dx<0,
$$ 
we get 
$
\sup_{t\geq 0}\zeta_\lambda(t)    
< c^\star
$
and the result follows. Otherwise, one has

\begin{equation*}
	\begin{split}
\sup_{t\geq 0}	\zeta_\lambda(t)    
\leq &\frac{s}{N}\left(\frac{S(\Sigma_{\mathcal{D}})}{ Q_M^{\frac{2}{2_s^*}}}\right)^{\frac{N}{2s}}+T_2^{q-1}\left[\frac{1}{2^*_s}T_{2}^{2^*_s-q+1}\int_{\Omega}(Q_M-Q(x))|\tilde{w}(x,0)|^{2^*_s}dx\right.\\&-\left. \frac{\lambda}{q+1} \int_\Omega |\tilde{w}(x,0)|^{q+1}dx\right]\\
<&\frac{s}{N}\left(\frac{S(\Sigma_{\mathcal{D}})}{ Q_M^{\frac{2}{2_s^*}}}\right)^{\frac{N}{2s}},
\end{split}
\end{equation*}
for  large $\lambda>0$. Then the conclusions of the Theorem can be drawn by the very same arguments as Theorem \ref{thmexistence}, simply replacing $\varphi_\lambda$ by $\zeta_\lambda$.
\end{proof}
\section{Multiplicity of solutions}\label{sec:multiplicity}
As pointed out in the Introduction, in the whole section we will make the standing assumption that \eqref{c=} holds. Let 
$$
N_\lambda:=\left\lbrace w\in \X\setminus\{0\}: \left\| w\right\|_{\X}^2-\int_\Omega Q(x)|w(x,0)|^{2^*_s}dx -\lambda\int_\Omega |w(x,0)|^{q+1}dx=0 \right\rbrace 
$$
be the Nehari manifold corresponding to $J_\lambda$. Having in mind to minimize $J_\lambda$ on suitable submanifolds of $N_\lambda$, we introduce first the barycenter map $\beta:\X\setminus\{0\}\to\R^N$ by
$$
\beta(w):=\frac{\displaystyle\int_\Omega x|w(x,0)|^{2^*_s}dx}{\displaystyle\int_\Omega |w(x,0)|^{2^*_s}dx}, \quad w\in \X\setminus\{0\}.
$$
By $(Q_2)$, there exists $r_0>0$ such that $\overline{B_{r_0}(a^i)}\cap\overline{B_{r_0}(a^j)}=\emptyset$ for all $i,j=1,\ldots,k$, $i\neq j$. Now, define
\[
	\begin{split}
	N_{\lambda,i} &:=\left\lbrace w\in N_\lambda:|\beta(w)-a^i|<r_0\right\rbrace,\\
	M_{\lambda,i} &:= \left\lbrace w\in N_\lambda:|\beta(w)-a^i|=r_0\right\rbrace.
	\end{split}
\]
The next results collects some basic properties of $N_\lambda$ and ${J_\lambda}_{|N_\lambda}$, whose proofs are similar to those of Lemma 3.1, 3.2 and 3.3 of \cite{lialiuzhatan2016existence}.

\begin{proposition}\label{propertiesNlambda}
The following facts hold: 
\begin{itemize}
	\item[$(i)$] $N_\lambda\neq\emptyset$ for all $\lambda\in(0,\lambda_1^s)$ if $q=1$; for all $\lambda>0$ if $q\in(1,2^*_s-1)$. In particular, for all $w\in\X\setminus\{0\}$, there exists in both cases a unique $t_w>0$, critical point of the map $t\mapsto J_\lambda(tw)$, such that $t_w w\in N_\lambda$;
	\item[$(ii)$] $J_\lambda$ is coercive and bounded below on $N_\lambda\neq\emptyset$ for all $\lambda\in(0,\lambda_1^s)$ if $q=1$; for all $\lambda>0$ if $q\in(1,2^*_s-1)$;
	\item[$(iii)$] $N_\lambda$ is a natural constraint for $J_\lambda$, i.e., any local extremal of $J_\lambda$ on $N_\lambda$ is a critical point on the whole $\X$.
\end{itemize}
\end{proposition}

\begin{proposition}\label{betaconvergence}
For every $i=1,\ldots,k$, let $t_{z_{\varrho,\varepsilon}^i}>0$ be such that $t_{z_{\varrho,\varepsilon}^i}z_{\varrho,\varepsilon}^i\in N_\lambda$. Then 
\begin{equation}\label{betaconv}
\beta(t_{z_{\varrho,\varepsilon}^i}z_{\varrho,\varepsilon}^i)\to a^i \quad \text{as } \varepsilon\to 0^+,
\end{equation} 
and hence there exists $\bar\varepsilon>0$ such that $\beta(t_{z_{\varrho,\varepsilon}^i}z_{\varrho,\varepsilon}^i)\in B_{r_0}(a^i)$ for every $\varepsilon\in(0,\bar\varepsilon)$ and $i=1,\ldots,k$.
\end{proposition}

\begin{proof}
One has
\begin{align*}
\beta(t_{z_{\varrho,\varepsilon}^i}z_{\varrho,\varepsilon}^i) &=\frac{\displaystyle\int_\Omega x z_{\varrho,\varepsilon}^i(x,0)^{2^*_s}dx}{\displaystyle \int_\Omega z_{\varrho,\varepsilon}^i(x,0)^{2^*_s}dx} =	\frac{\displaystyle\int_\Omega x \phi_0\left(\frac{|x-a^i|}{\varrho} \right)^{2^*_s}\frac{1}{(\varepsilon^2+|x-a^i|^2)^N} dx}{\displaystyle \int_\Omega \phi_0\left(\frac{|x-a^i|}{\varrho} \right)^{2^*_s}\frac{1}{(\varepsilon^2+|x-a^i|^2)^N}dx}\\
&=\frac{\displaystyle\int_\Omega (\varepsilon x+a^i) \phi_0\left(\frac{\varepsilon|x|}{\varrho} \right)^{2^*_s}\frac{\varepsilon^{2N}}{(\varepsilon^2+|\varepsilon x|^2)^N} dx}{\displaystyle \int_\Omega \phi_0\left(\frac{\varepsilon|x|}{\varrho} \right)^{2^*_s}\frac{\varepsilon^{2N}}{(\varepsilon^2+|\varepsilon x|^2)^N}dx}\\
&=\frac{\displaystyle\int_\Omega (\varepsilon x+a^i) \phi_0\left(\frac{\varepsilon|x|}{\varrho} \right)^{2^*_s}\frac{1}{(1+|x|^2)^N} dx}{\displaystyle \int_\Omega \phi_0\left(\frac{\varepsilon|x|}{\varrho} \right)^{2^*_s}\frac{1}{(1+|x|^2)^N}dx},
\end{align*}	
and therefore, taking the limit as $\varepsilon\to 0^+$, we obtain \eqref{betaconv}. The rest of the conclusion is clear.
\end{proof}

On account of Proposition \ref{propertiesNlambda}, it makes sense to define, for every $i=1,\ldots,k$, the quantities
\[
m_{\lambda,i}:=\inf_{w\in N_{\lambda,i}} J_\lambda(w), \quad \widetilde{m}_{\lambda,i}:=\inf_{w\in M_{\lambda,i}} J_\lambda(w).
\]

\begin{proposition}\label{mlipositive}
For every $i=1,\ldots,k$, then $m_{\lambda,i}>0$ for all $\lambda\in(0,\lambda_1^s)$ when $q=1$; for all $\lambda>0$ when $q\in(1, 2^*_s-1)$.
\end{proposition}

The proof of the proposition above follows the same arguments as \cite[Proposition 3.2]{lialiuzhatan2016existence}.

\begin{proposition}\label{tildemlambdai}
There exists $\widetilde\lambda>0$ such that
\[
\widetilde{m}_{\lambda,i}>\frac{s}{N}\cdot\frac{S(s,N)^\frac{N}{2s}}{2 Q_M^\frac{N-2s}{2s}}, 
\]
for all $\lambda\in(0,\widetilde\lambda)$ and for all $i=1,\ldots,k$.
\end{proposition}

\begin{proof}
	Arguing by contradiction, assume that there exists a sequence $\{l_n\}\subset\R^+$, $l_n\to 0^+$, such that
	$$
	\widetilde{m}_{l_n,i}\to\widetilde{m}\leq \frac{s}{N}\cdot\frac{S(s,N)^\frac{N}{2s}}{2 Q_M^\frac{N-2s}{2s}},
	$$
for some $i\in 1,\ldots,k$, and let $\{w_n\}\subset M_{l_n,i}$ such that $J_{l_n}(w_n)\to\widetilde{m}$ as $n\to\infty$. If $q=1$ we have
\begin{align*}
	\widetilde{m}+o(1)  = J_{l_n}(w_n) &=\left( \frac{1}{2}-\frac{1}{2^*_s}\right)\left\|w_n\right\|_{\X}^2+l_n \left( \frac{1}{2}-\frac{1}{2^*_s}\right)\int_\Omega w_n(x,0)^2 dx\\
	& > \left( \frac{1}{2}-\frac{1}{2^*_s}\right)\left(1-\frac{l_n}{\lambda_{1,s}}\right)\left\|w_n\right\|_{\X}^2
\end{align*}
and hence, since $l_n\to 0$, $\{w_n\}$ is bounded in $\X$. Likewise, if $q\in(1,2^*-1)$, we get
\begin{align*}
	\widetilde{m}+o(1)  = J_{\red{l_n}}(w_n)&=\left( \frac{1}{2}-\frac{1}{q+1}\right)\left\|w_n\right\|_{\X}^2+ \left( \frac{1}{q+1}-\frac{1}{2^*_s}\right)\int_\Omega Q(x) |w_n(x,0)|^{2^*_s} dx\\
	& > \left( \frac{1}{2}-\frac{1}{q+1}\right)\left\|w_n\right\|_{\X}^2,
\end{align*}
from which, again, the boundedness of $\{w_n\}$. Since $\{w_n\}\subset N_\lambda$ and
$$
l_n\int_\Omega Q(x) |w_n(x,0)|^{q+1}dx \to 0, \quad \text{as } n\to +\infty,
$$
up to subsequences
\[
\lim_{n\to +\infty}k_s\int_{\mathcal{C}_\Omega}y^{1-2s}|\nabla w_n|^2dxdy = \lim_{n\to +\infty} \int_\Omega Q(x)|w_{n}(x,0)|^{2^*_s}dx = l>0.
\]
By Sobolev embeddings we obtain
\begin{align*}
	l & = \lim_{n\to +\infty} \int_\Omega Q(x)|w_{n}(x,0)|^{2^*_s}dx \leq Q_M \lim_{n\to +\infty} \int_\Omega |w_{n}(x,0)|^{2^*_s}dx\\
	& \leq Q_M S(\Sigma_D)^{-\frac{2^*_s}{2}}\lim_{n\to +\infty}\left( k_s\int_{\mathcal{C}_\Omega}y^{1-2s}|\nabla w_n|^2dxdy\right)^{\frac{2^*_s}{2}}\\
	&=Q_M S(\Sigma_D)^{-\frac{2^*_s}{2}}l^{\frac{2^*_s}{2}},
\end{align*}
which yields
\begin{equation}\label{llowerbound}
l\geq Q_M^{-\frac{N-2s}{2s}}S(\Sigma_D)^{\frac{N}{2s}}.
\end{equation}
One also has
$$
\frac{s}{N}\cdot\frac{S(s,N)^\frac{N}{2s}}{2 Q_M^\frac{N-2s}{2s}}\geq \widetilde{m} = \lim_{n\to +\infty}	J_{l_n}(w_n)\\ 
= \left( \frac{1}{2}-\frac{1}{2^*_s}\right)l=\frac{s}{N}l, 
$$	
from which
\begin{equation}\label{lupperbound}
	l\leq \frac{S(s,N)^\frac{N}{2s}}{2Q_M^{\frac{N-2s}{2s}}}.
\end{equation}
Combining \eqref{llowerbound} and \eqref{lupperbound}, we obtain
\[
\frac{S(\Sigma_D)^\frac{N}{2s}}{Q_M^{\frac{N-2s}{2s}}}	\leq l\leq \frac{S(s,N)^\frac{N}{2s}}{2Q_M^{\frac{N-2s}{2s}}},
\]
and, on account of \eqref{c=},
$$
l=\frac{S(s,N)^\frac{N}{2s}}{2Q_M^{\frac{N-2s}{2s}}}.
$$
Moreover, one has
\begin{align*}
	\frac{S(s,N)^\frac{N}{2s}}{2Q_M^{\frac{N-2s}{2s}}} & \leq Q_M\lim_{n\to +\infty}\int_\Omega |w_n(x,0)|^{2^*_s}dx\leq Q_M S(\Sigma_D)^{-\frac{2^*_s}{2}}\left(\frac{S(s,N)^\frac{N}{2s}}{2Q_M^\frac{N-2s}{2s}} \right)^\frac{N}{N-2s}\\
	& =\frac{Q_M}{Q_M^\frac{N}{2s}}\cdot \frac{2^\frac{2s}{N-2s}}{2^\frac{N}{N-2s}}\cdot S(s,N)^{-\frac{N}{N-2s}+\frac{N^2}{2s(N-2s)}}\\
	&=\frac{S(s,N)^\frac{N}{2s}}{2Q_M^{\frac{N-2s}{2s}}}.
\end{align*}	
This forces 
\begin{equation}\label{QMwn}
\lim_{n\to +\infty}\int_\Omega Q_M |w_n(x,0)|^{2^*_s}dx= \frac{S(s,N)^\frac{N}{2s}}{2Q_M^{\frac{N-2s}{2s}}},
\end{equation}
and therefore
\begin{equation}\label{QMminusQ}
	\lim_{n\to +\infty}\int_\Omega (Q_M-Q(x))|w_n(x,0)|^{2^*_s}dx=0.
\end{equation}
Next, define $z_n:=\dfrac{w_n}{\left\| w_n(\cdot,0)\right\|_{2^*_s}}$. By \eqref{QMwn} we deduce that
$$
\lim_{n\to+\infty}\int_\Omega |w_n(x,0)|^{2^*_s} dx = \frac{S(s,N)^\frac{N}{2s}}{2Q_M^{\frac{N}{2s}}}
$$
and therefore
$$
\lim_{n\to+\infty} \left\| w_n(\cdot,0)\right\|_{2^*_s}^2 = \left( \frac{S(s,N)^\frac{N}{2s}}{2Q_M^{\frac{N}{2s}}}\right)^\frac{N-2s}{N} = \left( \frac{S(\Sigma_D)}{Q_M}\right)^\frac{N-2s}{2s}.
$$
So,
\begin{align*}
\lim_{n\to +\infty}k_s\int_{\mathcal{C}_\Omega}y^{1-2s}|\nabla z_n|^2 dxdy &=\lim_{n\to +\infty}\frac{\displaystyle k_s\int_{\mathcal{C}_\Omega}y^{1-2s}|\nabla w_n|^2 dxdy}{\left\| w_n(\cdot,0)\right\|_{2^*_s}^2}=\frac{S(\Sigma_D)^\frac{N}{2s}}{Q_M^\frac{N-2s}{2s}}\cdot \frac{Q_M^\frac{N-2s}{2s}}{S(\Sigma_D)^\frac{N-2s}{2s}}\\&=S(\Sigma_D),
\end{align*}
i.e., $\{z_n\}$ is a minimizing sequence for $S(\Sigma_D)$. Now, appealing to \cite[Theorem 4.5]{colort2019the}, two alternatives can occur.

i) $\{z_n\}$ is relatively compact. If so, up to a subsequence, we have $z_n\to z_0\in \X$ and $\left\| z_n\right\|_{\X}^2 \to \left\| z_0\right\|_{\X}^2$. But, from what seen before, $\left\| z_n\right\|_{\X}^2 \to S(\Sigma_D)= 2^{-\frac{2s}{N}}S(s,N)$, meaning that $S(s,N)$ is attained, a contradiction.

ii) $z_n\rightharpoonup 0$ and there exist a subsequence, still denoted by $\{z_n\}$ and $x_0\in\overline{\Sigma}_N$:
\begin{equation}\label{deltaconv}
	k_s y^{1-2s} |\nabla z_n|^2 \rightharpoonup \widetilde{S}(\Sigma_D)\delta_{x_0},\quad
	|z_n(\cdot,0)|^{2^*_s} \rightharpoonup \delta_{x_0}, \quad \text{as } n\to +\infty,
\end{equation}
in the sense of measures. So one has
$$
\beta(z_n)=\frac{\displaystyle\int_\Omega x |z_n(x,0)|^{2^*_s}dx}{\displaystyle\int_\Omega  |z_n(x,0)|^{2^*_s}dx}\to x_0,
$$
and since $|\beta(z_n)-a^i|=r_0$, by the continuity of $\beta$ it must be $x_0\neq a^i$. Finally, collecting \eqref{QMminusQ} and \eqref{deltaconv}, we arrive at $Q_M=Q(x_0)$, against the assumption $(Q_2)$. This ends the proof.
\end{proof}	

The following lemma is a direct consequence of the implicit function theorem and its proof relies upon the same arguments as \cite[Lemma 2.4]{tar1992on}.

\begin{lemma}\label{lemmatarantello}
Let $i\in {1,\dots,k}$. Then, for all $w\in N_{\lambda,i}$ there exists $R>0$ and a differentiable functional $\psi:B(0,R)\subset\X\to[0,+\infty)$ such that
\begin{align*}
	\psi(0) & =1,\quad \psi(z)(w-z)\in N_{\lambda,i}, \; \text{ for all } z\in B(0,R),\\
	\psi'(0)(z) & =\frac{2 k_s \displaystyle\int_{\mathcal C_{\Omega}}y^{1-2s}\nabla w\nabla z dxdy - 2^*_s\int_\Omega Q(x)|w(x,0)|^{2^*_s-1}z(x,0) dx}{\displaystyle(1-q)k_s\int_{\mathcal{C}_\Omega}y^{1-2s}|\nabla w|^2 dxdy - (2^*_s-q-1)\int_\Omega Q(x) |w(x,0)|^{2^*_s}dx}\\
	&\;\;\; -\frac{\displaystyle\lambda(q+1)\int_\Omega |w(x,0)|^q z(x,0)dx}{\displaystyle(1-q)k_s\int_{\mathcal{C}_\Omega}y^{1-2s}|\nabla w|^2 dxdy - (2^*_s-q-1)\int_\Omega Q(x)|w(x,0)|^{2^*_s}dx},\\ 
	&\;\;\; \text{for all } z\in C_0^\infty(\mathcal C_{\Omega}).
\end{align*}	
\end{lemma} 

Collecting the previous results, we can prove now the main multiplicity result for \eqref{problem}.

\begin{theorem}\label{thmksolutions}
Assume $(Q_2)$ where
\begin{align*}
	q\in \displaystyle\left( \frac{6s-N}{N-2s},\frac{N+2s}{N-2s}\right)  & \text{ if }  N=2,3,\\
	q\in\displaystyle \left[ 1,  \frac{N+2s}{N-2s}\right)  & \text{ if } N\geq 4.
\end{align*}
Then there exists $\lambda^\star>0$ such that \eqref{problem} has at least $k$ positive solutions for all $\lambda\in(0,\lambda^\star)$.
\end{theorem}

\begin{proof}
Thanks to Propositions \eqref{betaconvergence}, \eqref{tildemlambdai} and Proposition \eqref{supfiberingmap}, used here replacing $a^0$ and $z^0_{\varrho,\varepsilon}$ by $a^i$ and $z^i_{\varrho,\varepsilon}$, respectively, one has the following chain of inequalities
\begin{equation}\label{mlimtildeli}
m_{\lambda,i}\leq J_\lambda(t_{z_{\varrho,\varepsilon}^i}z_{\varrho,\varepsilon}^i)=\sup_{t\geq 0} J_\lambda(t z_{\varrho,\varepsilon}^i) < \frac{s}{N}\frac{S(s,N)^\frac{N}{2s}}{2Q_M^\frac{N-2s}{2s}}<\widetilde{m}_{\lambda,i},
\end{equation}
for every $\lambda\in(0,\widetilde\lambda)$ and, as a result,
$$
m_{\lambda,i}=\inf_{w\in N_{\lambda,i}\cup M_{\lambda,i}}J_\lambda(w),
$$
for every $\lambda\in(0,\widetilde\lambda)$. Let $\{w_{n,i}\}\subset N_{\lambda,i}\cup M_{\lambda,i}$ be a minimizing sequence for $m_{\lambda,i}$. Since $J_\lambda(|w|)=J_\lambda(w)$, we can assume that $w_{n,i}\geq 0$ a.e. in $\mathcal C_\Omega$. Our goal is to show that $\{w_{n,i}\}\subset N_{\lambda,i}\cup M_{\lambda,i}$ is a $(PS)_{m_{\lambda,i}}$-sequence for $J_\lambda$. By Ekeland variational principle, up to subsequences, one has
\begin{equation}\label{ekeland}
	\begin{split}
		J_\lambda(w_{n,i}) & <m_{\lambda,i} + \frac{1}{n},\\
		J_\lambda(w) & \geq J_\lambda(w_{n,i})-\frac{1}{n}\left\| w-w_{n,i}\right\|_{\X}, \quad \text{for all } w\in N_{\lambda,i}. 
	\end{split}
\end{equation} 
Moreover, by Lemma \ref{lemmatarantello}, there exist $R_{n,i}>0$ and $\psi_{n,i}:B(0,R_{n,i})\to[0,+\infty)$, differentiable, so that
$$
\psi_{n,i}(0)=1, \quad \psi_{n,i}(w)(w_{n,i}-w)\in N_{\lambda,i}, \text{ for all } w\in B(0, R_{n,i}).
$$
Choosing $\varphi\in \X$, $\left\| \varphi\right\|_{\X}=1$ and $\sigma\in(0, R_{n,i})$ 
by \eqref{ekeland} and the mean value theorem one gets, for some $t_0\in(0,1)$,
\begin{align*}
	&\frac{\left\| \psi_{n,i}(\sigma\varphi)(w_{n,i}-\sigma\varphi)-w_{n,i}\right\|_{\X}}{n}  \geq J_\lambda(w_{n,i}) - J_\lambda\left(\psi_{n,i}(\sigma\varphi)(w_{n,i}-\sigma\varphi)\right)\\
	& = J_\lambda'\left(t_0 w_{n,i}+(1-t_0)\psi_{n,i}(\sigma\varphi)(w_{n,i}-\sigma\varphi)\right)\left(w_{n,i}-\psi_{n,i}(\sigma\varphi)(w_{n,i}-\sigma\varphi)\right)\\
	& = J'_\lambda(w_{n,i})\left(w_{n,i}-\psi_{n,i}(\sigma\varphi)(w_{n,i}-\sigma\varphi)\right) + o\left(  \left\| w_{n,i}-\psi_{n,i}(\sigma\varphi)(w_{n,i}-\sigma\varphi)\right\|_{\X}\right)\\
	& =\left( 1-\psi_{n,i}(\sigma\varphi)\right) J'_\lambda(w_{n,i})(w_{n,i}) + \sigma\psi_{n,i}(\sigma\varphi)J'_\lambda(w_{n,i})(\varphi) \\
	& \;\;\; + o\left(  \left\| w_{n,i}-\psi_{n,i}(\sigma\varphi)(w_{n,i}-\sigma\varphi)\right\|_{\X}\right)\\
	& = \sigma\psi_{n,i}(\sigma\varphi)J'_\lambda(w_{n,i})(\varphi) + o\left(  \left\| w_{n,i}-\psi_{n,i}(\sigma\varphi)(w_{n,i}-\sigma\varphi)\right\|_{\X}\right). 
\end{align*}
Therefore we deduce
\begin{align*}
	|J'_\lambda(w_{n,i})(\varphi)| & \leq \frac{\displaystyle\left\|w_{n,i}-\psi_{n,i}(\sigma\varphi)(w_{n,i}-\sigma\varphi) \right\|_{\X} \left(\frac{1}{n}+|o(1)|\right)}{\sigma|\psi_{n,i}(\sigma\varphi)|}\\
	& = \frac{\displaystyle\left\|w_{n,i}\psi_{n,i}(0)-\psi_{n,i}(\sigma\varphi)w_{n,i} + \sigma\psi_{n,i}(\sigma\varphi)\varphi \right\|_{\X} \left(\frac{1}{n}+|o(1)|\right)}{\sigma|\psi_{n,i}(\sigma\varphi)|}\\
	&\leq \left( \left\|w_{n,i}\right\|_{\X}\frac{|\psi_{n,i}(\sigma\varphi)-\psi_{n,i}(0)|}{\sigma|\psi_{n,i}(\sigma\varphi)|} + 1 \right) \left(\frac{1}{n}+|o(1)|\right) 
\end{align*}
and taking the limit as $\sigma \to 0^+$, thanks to the boundedness of $\{w_{n,i}\}$ and $\{\psi'_{n,i}(0)\}$, we obtain ${J'_\lambda}(w_{n,i})\to 0$ as $n\to \infty$. So $\{w_{n,i}\}$ is a $(PS)$ sequence at the level $m_{\lambda,i}$ for $J_\lambda$. Now, setting $\lambda^\star:=\min\{\widetilde\lambda,\lambda_{1,s}\}$,  for every $\lambda\in(0,\lambda^\star)$, by \eqref{mlimtildeli} and Proposition \ref{PScondition}, $J_\lambda$ satisfies the $(PS)_{m_{\lambda,i}}$ condition and hence there exists $w_i\in\X$, $w_i\geq 0$ in $\mathcal C_\Omega$, such that $w_{n,i}\to w_i$ in $\X$ for every $i=1,\ldots,k$. Then, as a consequence of Ekeland variational principle and Propositions \ref{mlipositive} and \ref{propertiesNlambda}, $w_i$ is a non-negative and nontrivial solution of \eqref{problem} with energy $J_\lambda(w_i)=m_{\lambda,i}$ for every $i=1,\ldots,k$. The strong maximum principle forces $w_i>0$ and this concludes the proof.
\end{proof}	

\end{document}